\documentclass[11pt,a4paper,oneside]{paper}

\usepackage{amsmath,amsthm,amssymb,graphicx,lineno,geometry,color}
\usepackage[table,xcdraw]{xcolor}

\newtheorem{theorem}{\large \it Theorem}

\newtheorem{lemma}[theorem]{\it Lemma}

\newtheorem{definition}{Definition}

\newcommand{\N}{\mathbb{N}}
\newcommand{\R}{\mathbb{R}}

\newcommand{\Z}{\mathbb{Z}}

\newcommand{\W}{\mathcal{W}}

\newcommand{\bm}{\begin{bmatrix}}
\newcommand{\ex}{\end{bmatrix}}

\newcommand{\dbar}{\parallel}
\newcommand{\dcsd}{\textrm{dcsd}}

\newcommand{\ds}{\displaystyle}
\renewcommand{\phi}{\varphi}
\begin{document}
\title{Off-Lattice Random Walks with Excluded Volume: A New Method of Generation, Proof of Ergodicity and Numerical Results}
\author{Laura Plunkett (nee Zirbel)\\
  Department of Mathematics and Science,\\
  Holy Names University,\\
  3500 Mountain Blvd. \\
  Oakland, CA 94619, USA\\
  Email: \texttt{plunkett@hnu.edu}\\
  Homepage: \texttt{https://www.hnu.edu/faculty-staff/plunkett-laura}
  \and
  Kyle Chapman\\
  Mathematics Department\\ 
  University of California: Santa Barbara,\\
  South Hall \\
  Santa Barbara, CA 93106, USA\\
  Email: \texttt{klchapman@math.ucsb.edu}\\
  Homepage: \texttt{http://www.math.ucsb.edu/\textasciitilde klchapman}
  }
  \maketitle

\begin{abstract}
We describe a new algorithm, the reflection method, to generate off-lattice random walks of specified, though arbitrarily large, thickness in $\mathbb{R}^3$ and prove that our method is ergodic on the space of thick walks. The data resulting from our implementation of this method is consistent with the scaling of the squared radius of gyration of random walks, with no thickness constraint. Based on this, we use the data to describe the complex relationship between the presence and nature of knotting and size, thickness and shape of the random walk. We extend the current understanding of excluded volume by expanding the range of analysis of how the squared radius of gyration scales with length and thickness. We also examine the profound effect of thickness on knotting in open chains. We will quantify how thickness effects the size of thick open chains, calculating the growth exponent for squared radius of gyration as a function of thickness. We will also show that for radius $r\leq 0.4$, increasing thickness by $0.1$ decreases the probability of knot formation by 50\% or more.  
\end{abstract}

\section{Introduction}

Long strings of connected molecules, called polymers, are central structures in the life and physical sciences, as well as engineering. Prominent examples are DNA, proteins, polystyrene, and silicone. Many of the physical properties of polymers arise purely from the connectivity of their monomers rather than the chemical properties of the monomers themselves \cite{tangled}. Compounds made from the same chemicals with the same types of bonds but without the linearity of polymers fail to demonstrate the same special properties of polymers \cite{flory}. For example, polymeric liquids, such as chewing gum, dough and egg whites, demonstrate high viscosity and visible elasticity, and these characteristics are attributed to their string-like polymeric components \cite{DoiEdwards}.

Knotting plays a critical role in the unusual characteristics of these substances. With regard to DNA, Fiers and Sinsheimer first showed that DNA of bacteriophage $\varphi$X174 is a single-stranded ring \cite{Fiers62}. Because of this closure condition, knotting could potentially be captured in the spacial structure of the DNA. In 1976, Liu et al.\ discovered examples of knotted DNA, which was followed by the discovery of evidence for the existence of topoisomerases, enzymes which can knot and unknot DNA, suggesting that the topology of DNA has some fundamental functional role  \cite{Cham01, Liu76, wasserman1985}. Likewise, it has been demonstrated that rubber's elasticity is a product of the multitude of cross links that hold the polymer chains in position, rather than its complex chemistry. This too is the result of spatial entanglement, knotting and linking, of the polymeric components \cite{entangled, tangled}. While it is known that the topology of a ring polymer plays a critical role, affecting size, gel-electrophoretic mobility, resistance to mechanical stretching and behavior under spacial confinement, even basic characteristics of size and shape, including how the size of a macromolecule scales as a function of length, are not well understood for knotted ring polymers, though simulations are allowing exploration of this interaction  \cite{critical, Matsuda2003, orlandini1998, saka2008, tubiana2011}.  Understanding the complex relationship between shape, size and knotting in polymers could allow for the controlled production of new materials on the molecular level, rather than relying on trial and error \cite{tangled}.

Everyday experience suggests that knotting is not a phenomenon restricted to closed loops. Knotting in long, open chains is quite familiar and natural, although the historical mathematical study of knots has been concentrated on closed loops \cite{metastable}. Although open chains can be always be isotoped to an unknotted configuration and are therefore not {\it topologically} knotted, the physical characteristics of an open chain are strongly affected by the degree of entanglement, and we may consider these to be knotted \cite{micheletti2011}. For such open knots, also called knotted arcs, much of the difficulty lies in formalizing the definition of knotting and methods of knot identification \cite{orlandini}. Determining the knot type of an open chain often strongly depends on closure scheme, and may be a statistical quantity, differing greatly from the discrete definition of knotting for closed chains \cite{micheletti2011}. Being able to define when an arc is knotted can be utilized to study the degree of knot localization in closed loops \cite{tubiana2011}.

Polymer conformations are likewise affected by the characteristics of the solvents in which they are submersed. In the case of good solvent regimes, the polymer prefers contact with the solvent over self contact, and self repels \cite{degennes, DoiEdwards, rawdonbook}. With bad solvent regimes the polymer prefers self contact over solvent contact, which causes self attraction and leads to collapse or globular behavior \cite{degennes, DoiEdwards, rawdonbook}. Between these we have $\Theta$-solvents, where the attractive and repulsive forces are at equilibrium \cite{degennes, DoiEdwards, rawdonbook}. Given some polymer and solvent combination, the $\Theta$ temperature refers to the temperature at which the attractive and repulsive properties of the polymer are balanced, and the polymer behaves globally like an ideal chain, i.e. a random walk \cite{DoiEdwards}. Although our models do not include fluid dynamics or temperatures, we still use the language of solvent based behavior to describe these distinct classes of large scale behavior. 

Experiments can be done at the $\Theta$ point so that the global effects of excluded volume are eliminated \cite{flory}. Such polymers and related models have size characteristics with Gaussian statistics and have average radius of gyration proportional to $N^{1/2}$ where $N$ is the length of the chain \cite{degennes}. Many models exhibit this behavior, including Gaussian chains, ideal chains and freely rotating chains \cite{DoiEdwards}. 

We will consider a more general class of polymers in a good solvent regime, that is, polymer models that include some notion of thickness and self repulsion. Polymers are not arbitrarily thin, and their thickness imposes an excluded volume constraint: two segments of the chain are not allowed to occupy the same position in space or intersect. This interaction makes the average size of the polymers larger, as segments are forced to be further apart \cite{DoiEdwards, vologodskii}. These models are also called self-avoiding walks \cite{vologodskii}. 

The mathematical properties of self-avoiding walks are more complex than $\Theta$-condition walks \cite{DoiEdwards}. The interaction between segments of the chain makes exact calculation of physical properties prohibitively difficult \cite{DoiEdwards, flory,gans}. The mean square end-to-end distance of self avoiding walks has been estimated to scale approximately as $N^{1.2}$, as opposed to $N$ for ideal walks \cite{degennes, micheletti2011}. Corresponding experiments with actual self repelling polymers and numerical studies of self avoiding walks on the simple cubic lattice have shown that scaling can range between $N^{1.1}$ and $N^{1.2}$ \cite{degennes}.  While we will find more detailed experimental scaling values as a function of the thickness to expand these results, the fact remains that these walks are on average larger or ``swollen'' compared to their ideal counterparts \cite{flory}. Using microscopy techniques on DNA samples, Valle et al$.$ showed that the end-to-end distribution of long DNA molecules matches a pure self-avoiding walk distribution, suggesting that the study of models with excluded volume will be critical in understanding the behavior of many polymers \cite{valle}. 
 
In addition to the above effects of excluded volume, the introduction of thickness has very important implications to knotting and linking. Freedman, He and Wang proposed a ring conformation in 1994 that, while topologically unknotted, required increasing length or decreasing thickness to be continuously deformed to a standard ring \cite{freedman}. Such a conformation would be physically knotted, but not topologically knotted, and is commonly referred to as a Gordian knot \cite{gordian}. Stasiak and Pieranski applied the SONO algorithm, which attempts to move to the standard unknotted conformation through local isotopies, and found configurations that resisted such local deformations, strengthening the case for the existence of physical knots.\cite{gordian}. In 2012, Coward and Hass proved that Gordian links exist, finding a pair of thick rings, topologically unlinked, that cannot be continuously deformed to disjoint rings without increasing their lengths or decreasing their thickness \cite{coward}. These examples suggest that local, length preserving, isotopies cannot possibly connect the space of thick rings or links, and therefore that an adequately random generation method cannot rely on such local isotopies. 

These difficulties in studying the geometric and topological properties of polymer models with excluded volume emphasize the need for a robust generation method that will generate data for all possible thicknesses and lengths  while ensuring that conformations are being generated with the correct probabilities \cite{michels}. To this end, we want a generation method in which the likelihood of sampling a given configuration is the same for all possible configurations. We will prove that our generation method is ergodic, and therefore the steady state distribution is uniform over the set of configurations. We will numerically confirm that the averages and scaling of the squared radius of gyration with no thickness matches the theoretical averages and scaling of random walks. 

The most studied model of self-avoiding walks is lattice walks where no lattice point is visited more than once. These self-avoiding walks can be modified to model stiffer polymers by adding an energy penalty for consecutive segments that are not collinear \cite{micheletti2011}. The pivot algorithm is a Monte Carlo method to generate self-avoiding walks on the lattice, proposed by Lal in 1968 \cite{lal}. Given some self-avoiding walk on the simple cubic lattice, a new sample is generated as follows: a site along the walk, the pivot point, is chosen at random, and a random symmetry is applied to the part of the walk subsequent to the pivot point. The result is identical to the original walk before the pivot point, and the terminal segment is replaced a lattice symmetric version of the terminal segment. The resulting walk is accepted if it is self-avoiding. This method has been proven to be ergodic \cite{madrassokal}. The probability of accepting a move is order $N^{-0.19}$ for a walk of length $N$ \cite{madrassokal}. While this probability goes to zero, it does so slowly and allows for reasonable run times for simulations \cite{madrassokal}. Further, Madras and Sokal showed, in 1987, that any lattice Monte Carlo method that relied on local, length preserving moves would fail to be ergodic \cite{madras_ergo}. 

In 1988, Kleinin and Vologodskii used an ergodic method, which we will refer to as the benchmark method, to generate open and closed chains, off-lattice, with relatively small thicknesses (less than one tenth of the length of each segment) \cite{vologodskii}. They showed that even these very modest changes resulted in a dramatic decrease in knotting and linking compared with the ideal case, but the method could not supply sufficient data for very long or thick chains  \cite{vologodskii}.   

We implement an ergodic method where a configuration, with thickness, is modified by a reflection, through a random plane, of the terminal end. This is inspired by the generation via lattice symmetries in Lal's work, and is similar to Stallman and Gans' work performing rotational symmetries on off-lattice walks with fixed bending angles, and subsequent work by Pederson, Laso and Schutenberger \cite{lal,micelles,gans}. Unlike models on the lattice, the infinite number of possible off-lattice configurations complicates the discussion of ergodicity. As with the pivot method, we will show that these moves will connect the space of walks accommodating a specified thickness and that our method samples the state space correctly. This will allow us to generate much thicker off-lattice walks, extending Kleinin and Vologodskii's results about size and knotting in thick chains and rings to much thicker chains and rings. The data we generate will allow us to answer our most fundamental question: what is the relationship between the size, shape and knottedness of a linear polymer as a function of its thickness? 

\section{Self Avoiding Walks: Definitions and Notation}

Our polymer models will have as their primary structure a unit edge length random walk. We will be using the same definition for thickness as Millett, Piatek and Rawdon, with some simplifications due to the fact that our underlying structure is an equilateral random walk \cite{thickness}. Our open walks of $n$ edges will generically be called $W_n$, and will have vertices $v_0, ... ,v_n$, with $v_0$ the origin. Let $s_1, ... ,s_n$ denote the unit length edge segments of $W_n$, and let $\varphi_i$ be the angle between $s_i$ and $s_{i+1}$.

In order to define thickness, let us first define for each $x \in W_n$, where $x$ may be either a vertex or a point along an edge of $W_n$,  $d_x: W_n \rightarrow \R$ where $d_x(y) = \dbar x-y\dbar$. 

\begin{definition} As in \cite{thickness}, we call $y$ a turning point for $x$ if $y$ is a critical value for $d_x$, that is, $d_x$ changes from increasing to decreasing or from decreasing to increasing at $y$. The doubly-critical points of $W_n$ are defines as all pairs $x$ and $y$ in $W_n$ such that $x$ is a turning point for $d_y$ {\em and} $y$ is a turning point for $d_x$. Further, the minimum distance between all such doubly-critical pairs of points, taken over the entire walk $W_n$, is  $\dcsd(W_n)$, the doubly-critical distance for $W_n$.  
\end{definition}

Controlling $\dcsd(W_n)$ will constitute the long range constraint: doubly critical points of $W_n$ must be sufficiently far away from each other. We will also implement a corresponding angle or bending constraint: 

\begin{definition} Let $W_n$ be an $n$ edge random walk, and let $\varphi_i$ be the angle between $s_i$ and $s_{i+1}$. Let $\varphi(W_n)$ be the minimum of all of these angle, $\varphi(W_n) = \ds \min_{i = 1, ... , n-1} \varphi_i$.  
\end{definition}
 
Physically, the angle constraint is the local expression of being able to accommodate a tube of radius $r$, while $\dcsd(W_n)$ is the distal expression of accommodating the same tube. If we desire that if a walk accommodates a tube of radius $r$, then the angle between adjacent segments $s_i$ and $s_{i+1}$,  $\varphi_i$, should be sufficiently large, allowing two perpendicular disks of radius $r$ through the midpoints of the segments $s_i$ and $s_{i+1}$ to not intersect. This is equivalent to requiring that $\varphi_i > 2 \arctan(2r).$   
 
\begin{definition} Let $W_n$ be a random walk with position vectors $v_0$, $v_1$, ... $v_n$, with $v_0$ the origin. Then we say $W_n$ can accommodate a tube of radius $r$ if the following conditions are met: 
\begin{itemize} 
\item {\bf Long Range Interaction}: The distance between any doubly-critical pair of points is bounded below by $2r$: $\dcsd(W_n) > 2r$, and 
\item {\bf Short Range Interaction}: For adjacent edges, $s_i$ and $s_{i+1}$, if there were perpendicular disks centered at the midpoint of each segment, both of radius $r$, the two should not intersect. This is equivalent to saying that the angle between the two should be greater than $2\arctan(2r)$: $\phi(W_n)\geq 2\arctan(2r)$. 
\end{itemize}
We may also say that $W_n$ has thickness $r$ when it can accommodate a tube of radius $r$. Define $\W_{(n,r)}$ to be the space of all random walks, $W_n$, that can accommodate a tube of radius $r$.    
\end{definition}  

We note that while a walk $W_n$ may accommodate a tube of radius $r$, that may not be the biggest tube it can accommodate. A configuration that can accommodate a tube of radius $r$ can also accommodate a tube of any $r'$ with $r' < r$ as well, as the smaller tube is embedded in the larger, so the long range constraint is satisfied, and the angles between adjacent edges are larger than necessary, so the local constraint is also satisfied.  

While we are using the same definition and implementation of thickness as Millett, Piatek and Rawdon, there are numerous other methods for modeling walks with excluded volume, including self avoiding walks on the simple cubic lattice and variations, chains of hard spheres with either freely jointed or fixed angles and rod and bead models \cite{universal_knotting, madrassokal, thickness, micelles,gans}. (While the squared radius of gyration for these models might vary depending on the model, the scaling behavior is very similar, as we will discuss in Section 7.)  

\section{Reflection Method: Introduction}

Given some $r$ we wish to generate a random walk of length $n$, $W_n$, that can accommodate a tube of radius $r$, and we want this generation method to be capable of sampling {\it all} such walks. The reflection method samples the space of such walks by performing a random sequence of two moves: single and double reflections. 

A configuration is altered in a manner similar to Lal's pivot method: given some walk, we will choose a vertex along the chain and perform a reflection of the terminal end through a random plane through the chosen vertex. (In fact, for efficiency we will select a random plane so that the resulting reflection will have an allowable angle at the reflection site.) If we are performing a single reflection move, we will then check to see if the new walk supports a tube of radius $r$. If such a tube can be supported, we accept the new configuration. If it does not accommodate a tube of radius $r$, we throw out the altered chain and begin anew, with a new random vertex and a new random plane. A double reflection is similar, except we perform two random (allowable) reflections, through different random vertices, before checking if our thickness conditions have been violated. Below are more detailed descriptions of each type of move. 

We ensure that this generation method samples all configurations in the space of walks accommodating a certain thickness, $r$. This is equivalent to showing that there is a finite sequence of possible reflection moves that takes us from any configuration to any other, thereby connecting the space of all walks with thickness $r$.    

\subsection{Single Reflection Moves}
 
Beginning with a configuration $W_n$ that accommodates a tube of radius $r$, we perform a reflection move by selecting a vertex, $v_i$, with $i \neq 0,n$, at random. Then we select a random (allowable) plane through $v_i$, $P_i$. We reflect the points $v_{i+1}$, $v_{i+2}$, ... , $v_n$ through this plane to obtain the points $\hat{v}_{i+1}$, $\hat{v}_{i+2}$, ... , $\hat{v}_n$. Then we consider the new $\hat{W}_n$ consisting of position vectors $v_0$, ... $v_i$, $\hat{v}_{i+1}$, $\hat{v}_{i+2}$, ... , $\hat{v}_n$, as in Figure \ref{Single_EX}. If $\hat{W}_n$ can accommodate a tube of radius $r$, then $\hat{W}_n$ is accepted as the new configuration. Otherwise, the attempt fails, and we try another reflection move on $W_n$.

\begin{figure}[!htp]
  \begin{center}
    \includegraphics[scale=.8]{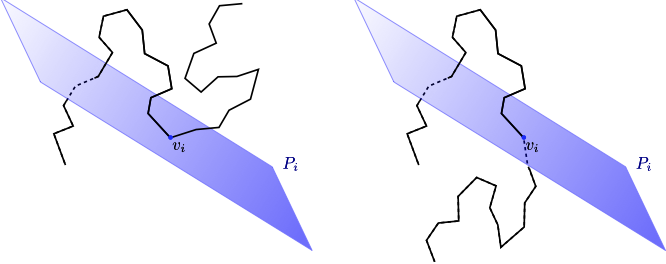}
  \end{center}
 \caption{A reflection through the plane $P_i$ of the terminal end of the chain.}
\label{Single_EX}
\end{figure}

\subsection{Double Reflection Moves} 

A double reflection move consists of performing two reflections. Selecting two random vertices, $v_i$ and $v_j$, with $i<j$, we also select random planes, $P_i$ and $P_j$, through these points. Again, for efficiency in simulation, we will only select planes such that the subsequent reflections will not violate the short range constraints for our given $r$. As with a single reflection moves, we reflect the points $v_{i+1}, ..., v_n$ through the plane $P_i$ obtaining $\hat{v}_{i+1}$, $\hat{v}_{i+2}$, ... , $\hat{v}_n$. Then we perform a second reflection through the plane $P_j$ of the points $\hat{v}_{j+1}$, $\hat{v}_{j+2}$, ... , $\hat{v}_n$ obtaining $\hat{\hat{v}}_{j+1}$, $\hat{\hat{v}}_{j+2}$, ... , $\hat{\hat{v}}_n$, as in Figure \ref{Double_EX}. Our new configuration, $\hat{\hat{W}}_n$ has vertices $$v_0,...,v_i, \hat{v}_{i+1}, ..., \hat{v}_j, \hat{\hat{v}}_{j+1}, ...,\hat{\hat{v}}_n.$$ If $\hat{\hat{W}}_n$ can accommodate a tube of radius $r$, $\hat{\hat{W}}_n$ is accepted as the new configuration. Otherwise, the attempt fails, and we try another reflection move on $W_n$.

\begin{figure}[!htp]
  \begin{center}
    \includegraphics[scale=.8]{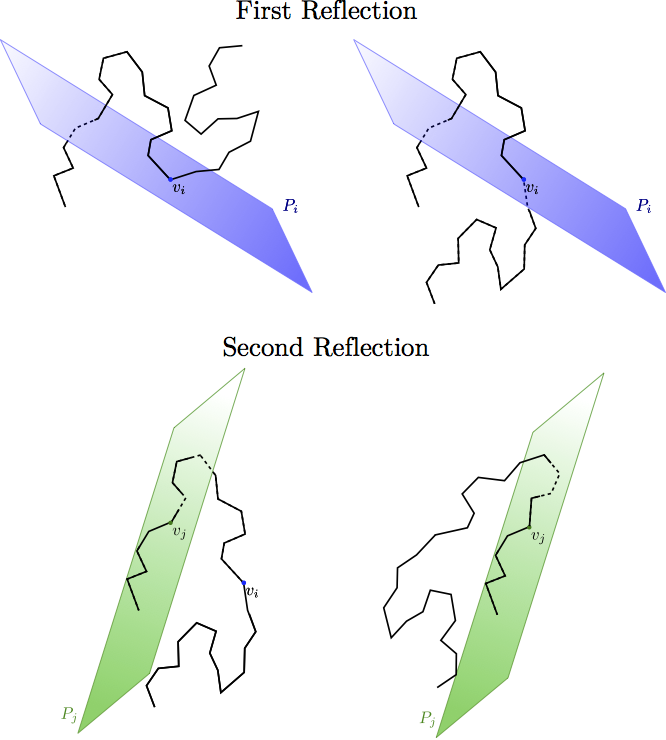}
  \end{center}
 \caption{A double reflection, first through the plane $P_i$, then through the plane $P_j$. }
\label{Double_EX}
\end{figure}

\section{Ergodicity of the Reflection Method}

We let the space of all walks be \(\W_{n} = \prod_n S^2\). We call the space of walks with thickness greater than or equal to \(r\), \(\W_{(n,r)}\). We form a Markov chain with state space \(\W_{(n,r)}\) and noise parameter \(X = (\Z_{n-1}\times \R P^2)^2\) with uniform probability distribution. This noise parameter is a pair of reflections, with the \(\Z_{n-1}\) coordinate indicating at which vertex to break the curve and the \(\R P^2\) coordinate determining the plane of reflection. We note that one of these reflections may be the trivial reflection, and therefore this encompasses single reflections as well. The Markov function \(F:\W_{(n,r)}\times X\rightarrow \W_{(n,r)}\) is then applying the pair of reflections in the \(X\) input to the walk in the \(\W_{(n,r)}\) coordinate. If the result is in \(\W_{(n,r)}\) it is kept, and if not we return the input walk.

A Markov chain \(F\) is forward accessible if for every starting point \(W_n\), the set of points in the state space \(\W_{(n,r)}\) which can be reached from \(W_n\) using a path on the interior of a smooth section has non-empty interior.

A Markov chain naturally creates a family of probability distributions \(P_m(x,A)\) which is the probability of landing in \(A\) after starting at \(x\) and going exactly \(m\) steps. From those we can take any distribution \(a(m)\) on the natural numbers and build transition kernels \(\kappa_a(x,A) = \displaystyle \sum_{m\in \N} P_m(x,A)a(m)\). The Markov chain \(F\) is a \(T\)-chain if there is a continuous piece \(T\) to some transition kernel \(\kappa_a\). Such a function \(T\) is a continuous piece if \(T \leq \kappa_a\), \(T(x,W_{(n,r)}) \neq 0\) for all \(x\) and \(T(\cdot,A)\) is lower semi-continuous for every \(A\). The last of these is why it is called a continuous piece.

A sequence of probability distributions \(\mu_n(Y)\) is tight if for every \(\epsilon > 0\) there exists a compact set \(C\) with \(\liminf_{n\rightarrow \infty}\mu_n(C) > 1-\epsilon\). When dealing with a Markov chain, we ask if a Markov chain is \emph{bounded in probability on average}, which means the sequence of probability distributions \(\overline{P_n(x,\cdot)} := \displaystyle \frac{1}{k}\sum_{n=1}^{k} P^n(x,\cdot)\) is tight. Since our state space \(\W_{(n,r)}\) is compact, any sequence of probability distributions is tight, simply by taking \(C = \W_{(n,r)}\).

A Markov chain is Harris if every set with positive borel measure is expected to be reached an infinite number of times regardless of the starting position. F is positive if there is a probability measure \(\mu\) on the state space which is invariant under iteration by the Markov move. These two properties which are often used together are referred to as positive Harris.

A Markov chain is periodic with period \(d\) if there exists a collection of \(d\) disjoint non-empty closed sets \(C_i\) with the probability of going from \(C_i\) to \(C_{i+1}\) is \(1\) for every \(i\). A Markov chain is aperiodic if the only period is \(1\).

\begin{theorem}
Given any walk \(W_n\) in \(\W_{(n,r)}\) there exists a finite sequence of double reflections (and single reflections) $\rho_i$ which take \(W_n\) to the straight walk, and with each reflection $\rho_i$ being on the interior of a neighborhood of moves \(R_i\) whose images are all in \(\W_{(n,r)}\). Thus, the finite sequence taking \(W_n\) to the straight walk is on the interior of a smooth section of \(F\).
\end{theorem}

\begin{proof}
Beginning with a random walk $W_n$ accommodating a tube of radius $r$, we will construct a finite sequence of single and double reflection moves taking $W_n$ to the straight configuration such that each intermediate step can also accommodate a tube of radius $r$. The construction will consist of three steps: 
\begin{itemize}
\item The first step is to use single reflections until the diameter of the convex hull is determined by the first and last vertex $v_0$ and $v_n$. 
\item Secondly, we will use double reflections to eliminate critical points (with respect to the axis determined by the diameter of the convex hull.)
\item Lastly, with a configuration that is strictly increasing with respect to some axis, we will straighten inductively using double reflection moves.  
\end{itemize}  

The following definition will be helpful in the first step outlined above.  

\begin{definition} Let the convex hull of $W_n$ be denoted $\textrm{\em Hull}(W_n)$. Let $v_i$ be some point on the surface of $\textrm{\em Hull}(W_n)$, with $i \neq 0$ and $i \neq n$. Let $Q_i$ be a plane through $v_i$ such that $Q_i \cap \textrm{\em int}(\textrm{\em Hull}(W_n)) = \emptyset$. (This is equivalent to saying that  $Q_i$ is incident with $\textrm{\em Hull}(W_n)$ at a point, edge or face of the convex hull containing $v_i$.) We may define a reflection through this plane of the points $v_{i+1}$, $v_{i+2}$, ... $v_n$ obtaining $\hat{v}_{i+1}$, $\hat{v}_{i+2}$, ... $\hat{v}_n$, as in Figure \ref{reflection1}. We define the {\em half reflection through $Q_i$} of $W_n$ to be the new random walk $\hat{W}_n$ consisting of the points $v_0$, $v_1$, ... , $v_i$, $\hat{v}_{i+1}$, $\hat{v}_{i+2}$, ... $\hat{v}_n$. 
\end{definition} 

\begin{figure}[!htp]
  \begin{center}
    \includegraphics[scale=1]{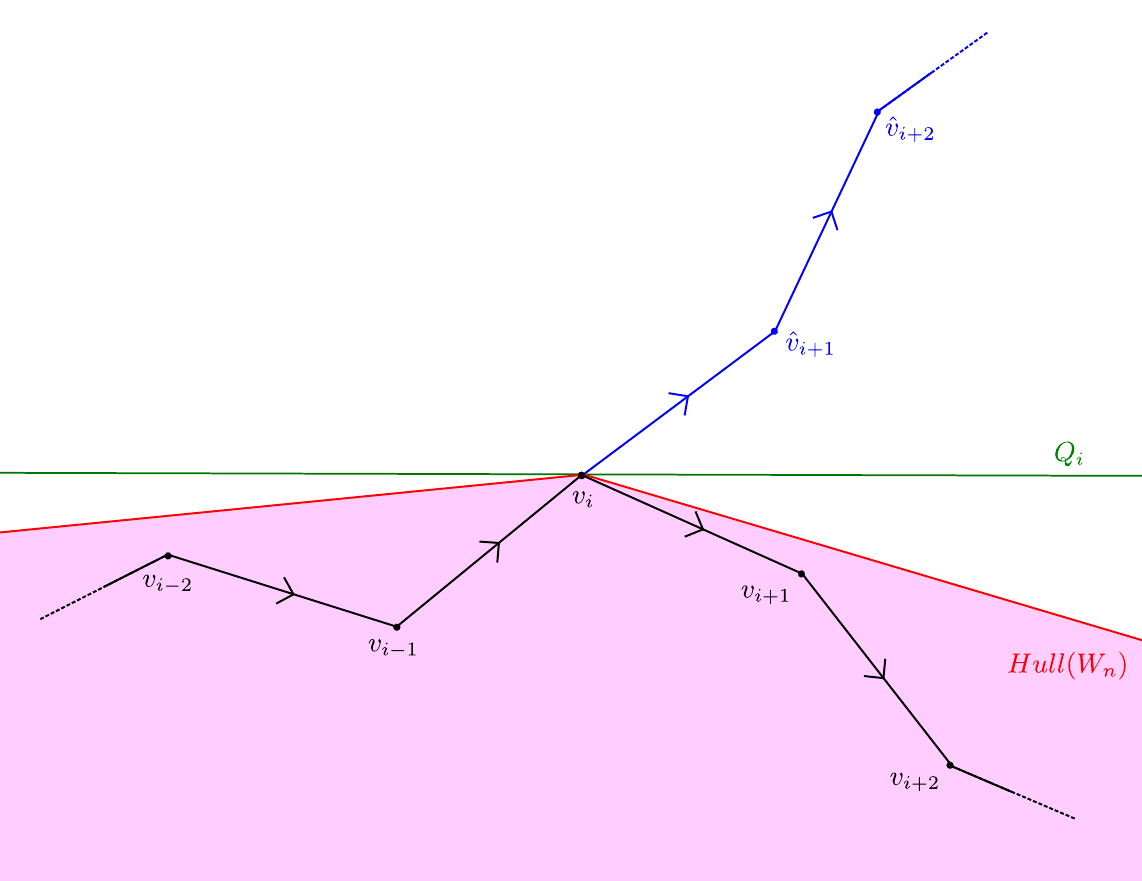}
  \end{center}
 \caption{Here we have a typical half reflection through a plane $Q_i$ incident with a point $v_i$ on the convex hull.}
    \label{reflection1}
\end{figure}

\begin{lemma}
If $W_n$ can accommodate a tube of radius $r$, then so can $\hat{W}_n$, where $\hat{W}_n$ is the result of the half reflection through $Q_i$ described above.   
\end{lemma}

\begin{proof}
There are two concerns with regard to the new thickness. The first is the long range interaction: is the distance between any doubly-critical pair of points of $\hat{W}_n$ less than $2r$ after such a reflection? The second is short range interaction, or curvature: have the angles between adjacent edges become smaller than $2\arctan(r/2)$? 

Let us address the simpler case of short range concern first. As the first $i$ vertices are unchanged, we are assured that $\phi_0$ to $\phi_{i-1}$ (where $\phi_j$ is the angle between $s_{j}$ and $s_{j+1}$) are unchanged, and therefore do not violate my curvature condition. Likewise, as reflection through $Q_i$ is an isometry on the segment defined by vertices $v_i$, $v_{i+1}$, ... $v_n$, the angles $\hat{\phi}_{i+1}$ to $\hat{\phi}_n$ are equal to their counterparts, $\phi_{i+1}$ to $\phi_n$, and therefore also unchanged. Thus the only problematic location is at $v_i$.

Without loss of generality, we may assume that $v_i$ is the origin, and $Q_i$ is the $yz$-plane. Then we have two unit length vectors, $(a,b,c)$ and $(d,e,f)$, rather than $s_{j}$ and $s_{j+1}$. Because we also know that $s_{j}$ and $s_{j+1}$ are on the same side of $Q_i$, we may further assume that the $x$ component of both points is positive, that is, $a, d > 0$. We have diagrammed this scenario in Figure \ref{trap}. 

\begin{figure}[!htp]
  \begin{center}
    \includegraphics[scale=1]{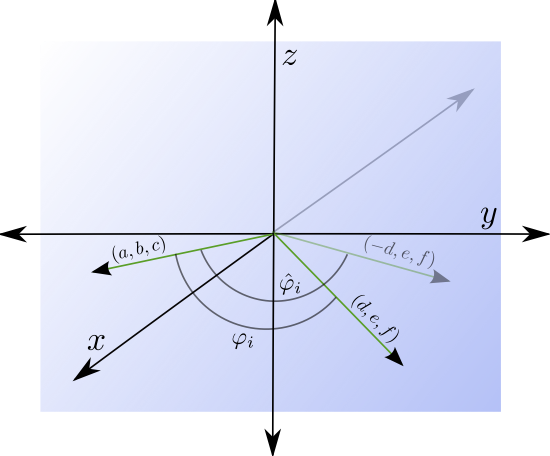}
  \end{center}
 \caption{We claim that the new angle $\hat{\phi}_i$ is greater than $\phi_i$.}
    \label{trap}
\end{figure}

We know that $\cos(\varphi_i) = ad + be + cf$ and that $\cos(\hat{\varphi}_i) = -ad + be + cf$. Because $a,d > 0$, we have that $\cos(\varphi_i) > \cos(\hat{\varphi}_i)$ and therefore $\varphi_i < \hat{\varphi}_i$, so we have not violated our short range conditions.  

Now let us consider the long range condition. As with the short range interaction, no two edges from the subsegment consisting of points $v_0$, ..., $v_i$ violate the long range interaction constraint as these edges are unchanged by the reflection. Likewise the subsegment $v_i$, $\hat{v}_{i+1}$, ..., $\hat{v}_n$, has only been reflected and therefore pairwise distances have been preserved. What remains to show is that no point from the first subsegment is within $2r$ of any point from the second subsegment. 

Suppose not, then two points, $a \in s_h$ and $b \in s_j$ with $h < i < j$ have become a too close pair of doubly critical points following a reflection of this type, that is, $|a-\hat{b}|<2r.$ Define the pill of a segment $s_i$ to be the set of points within $r$ of $s_i$, that is $$P_i = \{z \in \R^3 | d(z,s_i)< r\}.$$ Then we have two cases to consider:
\begin{itemize} 
\item $P_h \cap \hat{P}_j = \emptyset.$ The points $a$ and $b$ are closer than the points $a$ and $\hat{b}$. Therefore, if $P_h \cap \hat{P}_j = \emptyset$, $a$ and $\hat{b}$ must be further than $2r$ from each other and cannot violate the long range thickness condition, as in Figure \ref{pills}.     
\begin{figure}[!htp]
  \begin{center}
    \includegraphics[scale=.45]{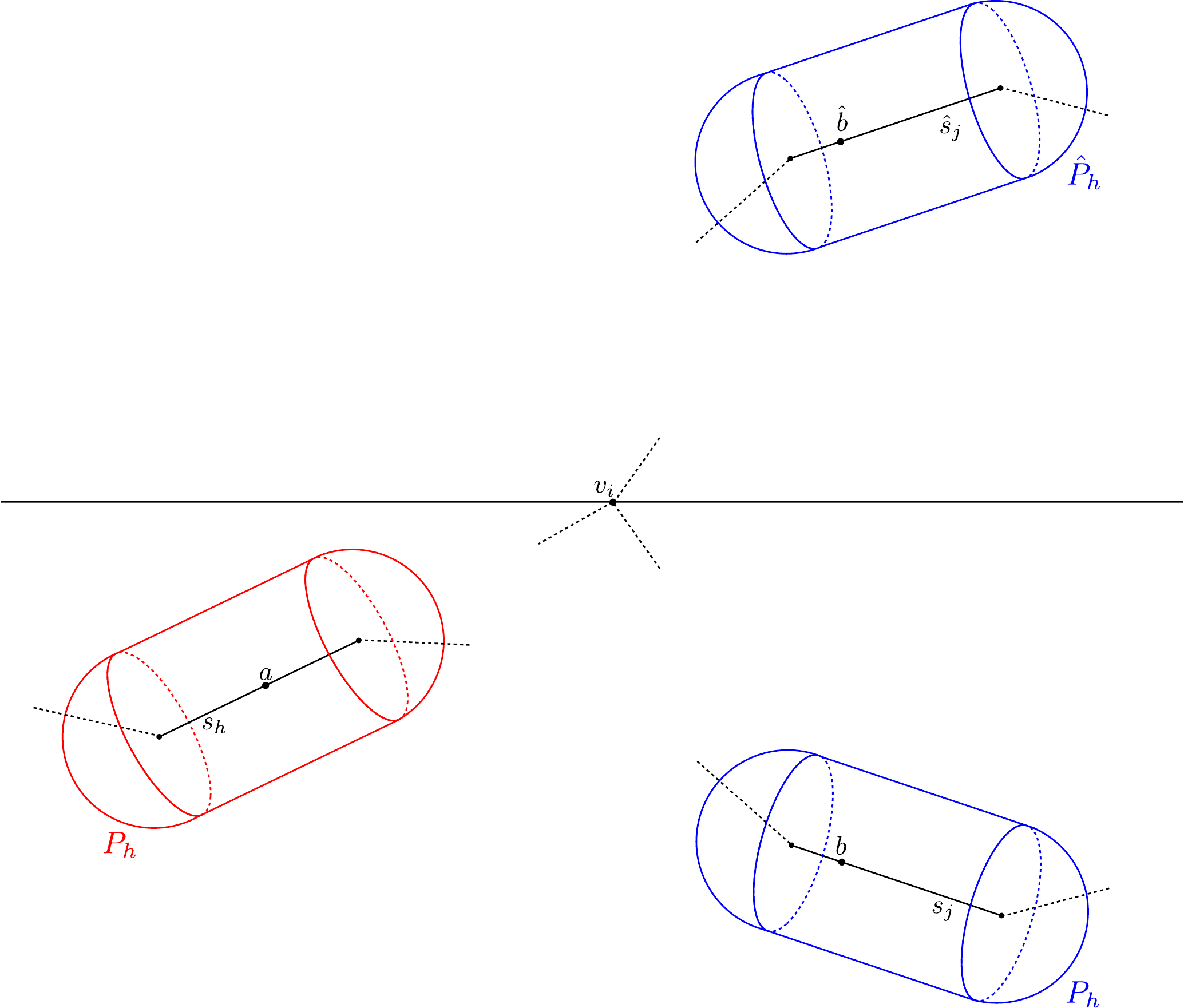}
  \end{center}
 \caption{If $P_h$ and $P_j$ do not intersect, then $P_h$ and $\hat{P}_j$ also do not intersect. }
    \label{pills}
\end{figure}

\item $P_h \cap \hat{P}_j \neq \emptyset.$ As above, if $a$ and $\hat{b}$ are not within $2r$ of each other, we have no issue. So assume that $d(a,\hat{b})\leq 2r$. In order for $a$ and $\hat{b}$ to become doubly critical, there is a minimum number of edges that must be between between $s_h$ and $\hat{s}_j$, namely, half the number of edges in a regular polygon with interior angles greater than or equal to $2\arctan(2r)$, as in Figure \ref{polygon}. Regular polygon geometry leads to this number being $\left\lfloor \ds \frac{\pi}{\pi-2\arctan(2r)}\right\rfloor.$ We conclude that there must be at least this many edges between $s_h$ and ${s}_j$. 

\begin{figure}[!htp]
  \begin{center}
    \includegraphics[scale=.1]{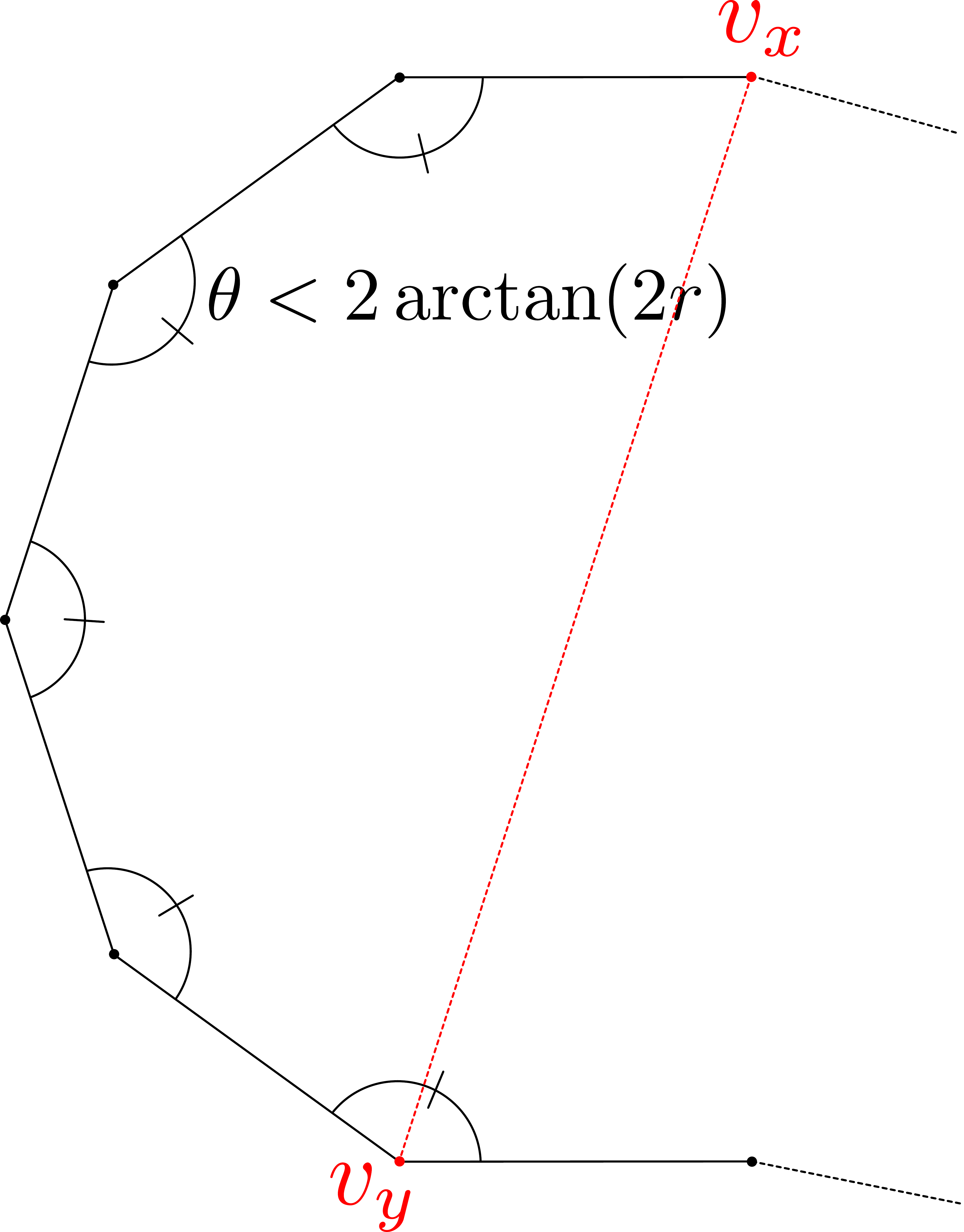}
  \end{center}
 \caption{There must be at least $\left\lfloor \ds \frac{\pi}{\pi-2\arctan(2r)}\right\rfloor$ edges between a pair of doubly critical points $v_x$ and $v_y$.}
    \label{polygon}
\end{figure}

Because $a$ and $b$ are within $2r$ of each other, they are therefore not a pair of doubly critical points, and we may assume that for every point along the segment between $a$ and $b$, the distance from $a$ to $b$ is strictly increasing. Likewise for the distance from $b$ to $a$, by symmetry. This confines the segment, and it must be contained in the union of the two spheres of radius $d(a,b)$ centered at $a$ and $b$. We will consider how many edges can fit between $a$ and $b$ in this union.  

\begin{figure}[!htp]
  \begin{center}
    \includegraphics[scale=.15]{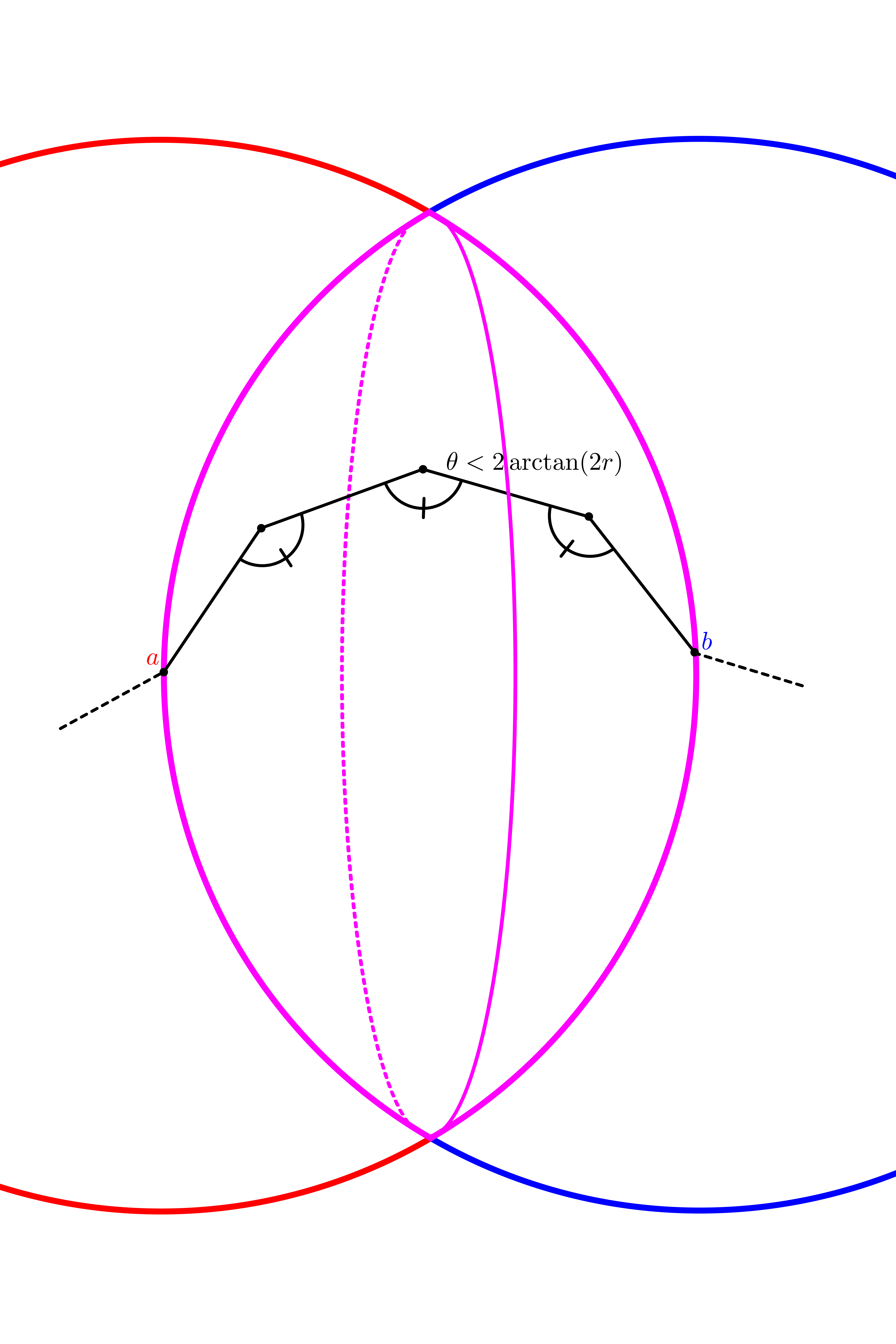}
  \end{center}
 \caption{For the pills $P_h$ and $P_j$ to intersect, without $a$ and $b$ being doubly critical, the maximum number of edges between them is $ \ds \frac{2\arctan(2r)}{\pi - 2\arctan(2r)}$. }
    \label{spheres}
\end{figure}

The most edges one can fit into this union of spheres without violating the short range (angle) restriction or having a doubly critical pair will be a segment of a regular $n$-gon where the internal angle $\theta$ is greater than or equal to $2\arctan(2r)$, as in Figure \ref{spheres}. The diagonal of this arc, which we will call $d_m$, must be less than $d(a,b)$ (and therefore also less than $2r$) in order to be a segment between $a$ and $b$. From regular polygon geometry, we know that $d_m = d\sin\left(\ds \frac{\pi m}{n}\right)$ where $m$ is the number of edges connecting the endpoints of the diagonal determined by $d_m$, and $d$ is the diameter of the $n$-gon. We can show that the diameter of this $n$-gon is $\sqrt{4r^2+1}$ and that $n = \frac{2\pi}{\pi-\theta}$. Then we find 
\begin{eqnarray*}
m& < & \ds \frac{2}{\pi - 2\arctan(2r)}\arcsin\left(\ds \frac{2r}{\sqrt{4r^2+1}}\right)\\
& <  & \ds \frac{2}{\pi - 2\arctan(2r)}\arctan(2r).\\
\end{eqnarray*}

Therefore, we know that the number of edges between $s_h$ and $s_j$ is $m$ and that $$\ds\frac{2\arctan(2r)}{\pi-2\arctan(2r)} >m>\left\lfloor \ds \frac{\pi}{\pi-2\arctan(2r)}\right\rfloor.$$
This is a contradiction, as $2\arctan(2r) \not> \pi$ for any $r$. Therefore no such doubly critical pair can exist. 
\end{itemize}

We conclude that $s_{k}$ and $\hat{s}_{j}$ do not violate our long range interaction condition for any $a$ and $b$. Therefore if $W_n$ can accommodate a tube of radius $r$, then $\hat{W}_n$ can as well. In fact, we can say more. Such a reflection has a neighborhood of similarly valid reflections, namely reflections through other planes incident with the convex hull at that point. This will be used in the proof of ergodicity. 

\end{proof}

\begin{lemma}
\label{endsout}
Suppose $W_n$ can accommodate a tube of radius $r$, and that $v_i$ and $v_j$ determine the diameter of $Hull(W_n)$. Then if one of these points is not an end point, that is, $i \notin \{0,n\}$ or $j\notin\{0,n\}$, we may perform a single reflection move that simultaneously increases the diameter by at least $\sqrt{d^2+1} - d$, where $d$ is the original diameter, and does not decrease the maximum tube radius the configuration can accommodate. 
\end{lemma}

\begin{figure}[!htp]
  \begin{center}
    \includegraphics[scale=.3]{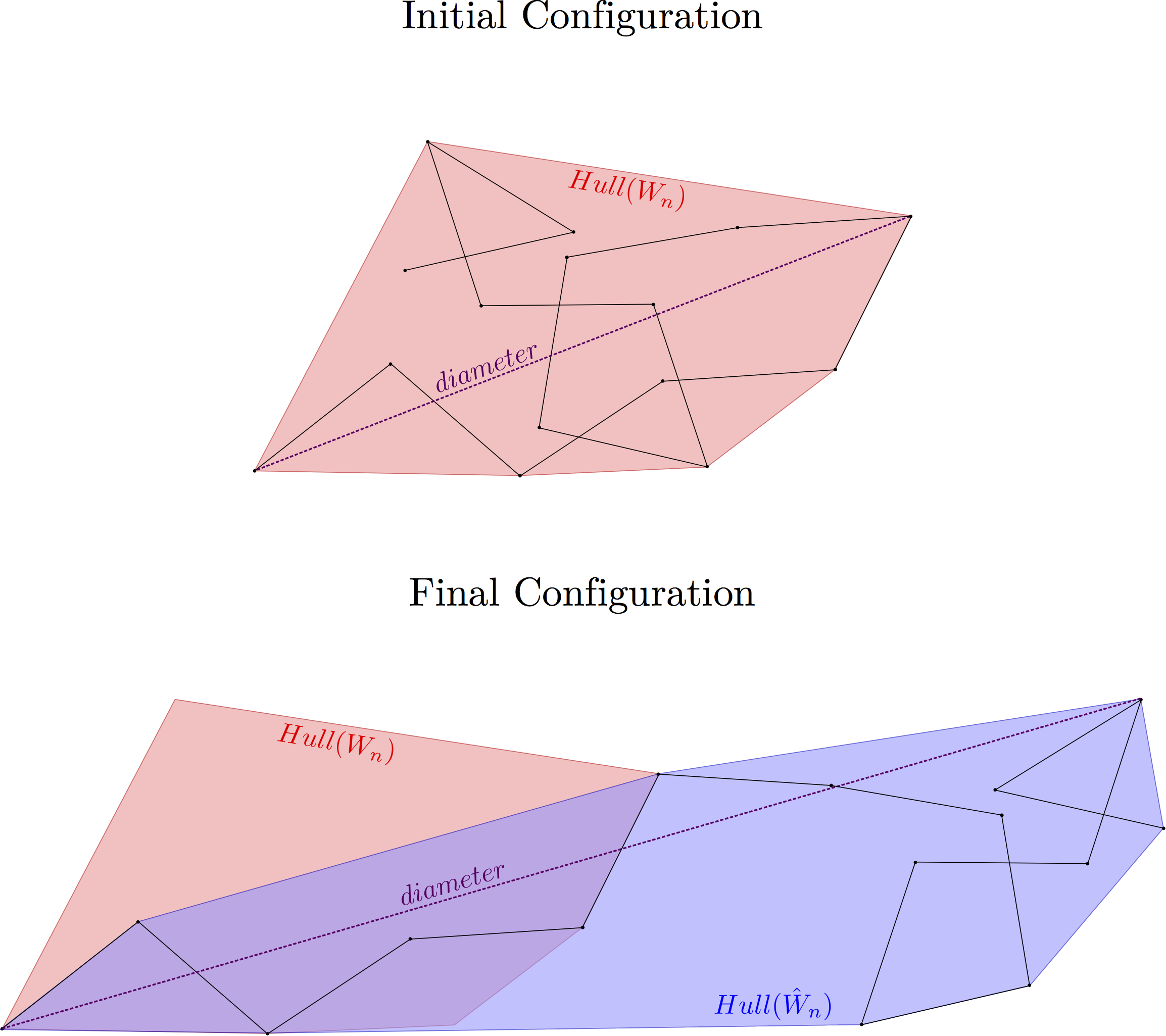}
  \end{center}
 \caption{After a reflection, the diameter of the convex hull has increased. In red, we have the convex hull of $W_n$, and in blue we have the convex hull of $\hat{W}_n$, resulting from a reflection through the right endpoint of the original diameter.}
    \label{endout}
\end{figure}

\begin{proof}
Consider a random walk $W_n$ with position vectors $v_0$, $v_1$, ... $v_n$ and convex hull $\textrm{Hull}(W_n)$. Let $v_i$ and $v_j$ be the points that determine the diameter of the convex hull, with length $d(v_i,v_j)$. Suppose $v_i$ is not an end of the random walk, that is, $i \neq 0,n$. Then make the half reflection through some suitable $Q_i$. We now may consider the convex hull of the first $i$ vertices, joined with the convex hull of the last $n-i$ vertices at $v_i$. The diameter of one of these is the same as the original diameter, $d(x_i,x_j)$. The other is non-zero, and in fact greater than or equal to $1$. Therefore the diameter of the new convex hull is greater than or equal to $\sqrt{d^2 + 1}$, as in Figure \ref{endout}.  
\end{proof} 

We may conclude, from Lemma \ref{endsout}, that we may proceed in this manner, increasing the length of the diameter of the convex hull until we have the two ends of the random walk determining the diameter of the convex hull, as the diameter is bounded by $n$. We call this specific diameter $d^*$. We now wish to consider the altered conformation $W_n'$, and all following confirmations, in relation to this axis, letting the $v_0$ end be the ``lower part'' and the $v_n$ end be the ``upper part.'' Any maxima and minima $W_n$ has with respect to this axis will be the focus of the second part of the proof.

If $W_n$ has no local maxima or minima with respect to the axis, we may move to the third part of the proof. 

If $W_n$ does have maxima or minima with respect to this axis, then we will be looking at very specific sites and performing double reflections. Consider some local minima or maxima of our walk, which may be either a point $v_i$ or an edge $s_i$. There is a plane through it and perpendicular to the original diameter $d^*$, which we will call $P_i$, as before. 

Specifically, find the local maxima $m_i$ such that the distance between its plane $P_i$ and the plane $P_n$ is minimized, as in Figure \ref{layers1}. Between $v_i$ and $v_n$ there is at least one, and possibly many local minima. We will select the local minima $n_j$, $i<j<n$, such that the distance between the plane containing $n_j$, $P_j$, and the plane containing $v_n$, $P_n$, is greatest, see Figure \ref{layers1}. 

\begin{figure}[!htp]
  \begin{center}
    \includegraphics[scale=1]{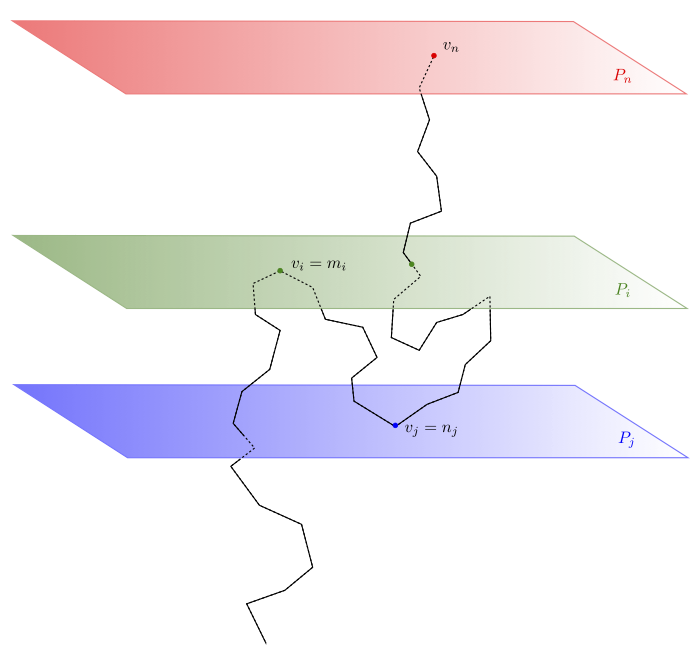}
  \end{center}
 \caption{After finding the local maxima $m_i$ and plane $P_i$, we find $n_j$ and $P_j$, the local minima between $v_i$ and $v_n$ furthest from $P_n$.}
    \label{layers1}
\end{figure}

First, we will reflect the segment from $v_j$ to $v_n$ through the plane $P_j$, obtaining $\hat{v}_{j+1}$, ... , $\hat{v}_n$. Then we have a new configuration consisting of position vectors $v_0, ..., v_j,\hat{v}_{j+1}, ..., \hat{v}_n$. Now we reflect the vectors $v_i, ..., v_j,\hat{v}_{j+1}, ..., \hat{v}_n$ through the second plane, $P_i$, as in Figure \ref{dr}, obtaining a new configuration with position vectors $$v_0, ..., v_i, \hat{v}_{i+1}, ..., \hat{v}_j,\hat{\hat{v}}_{j+1}, ..., \hat{\hat{v}}_n.$$

\begin{figure}[!htp]
  \begin{center}
    \includegraphics[scale=.6]{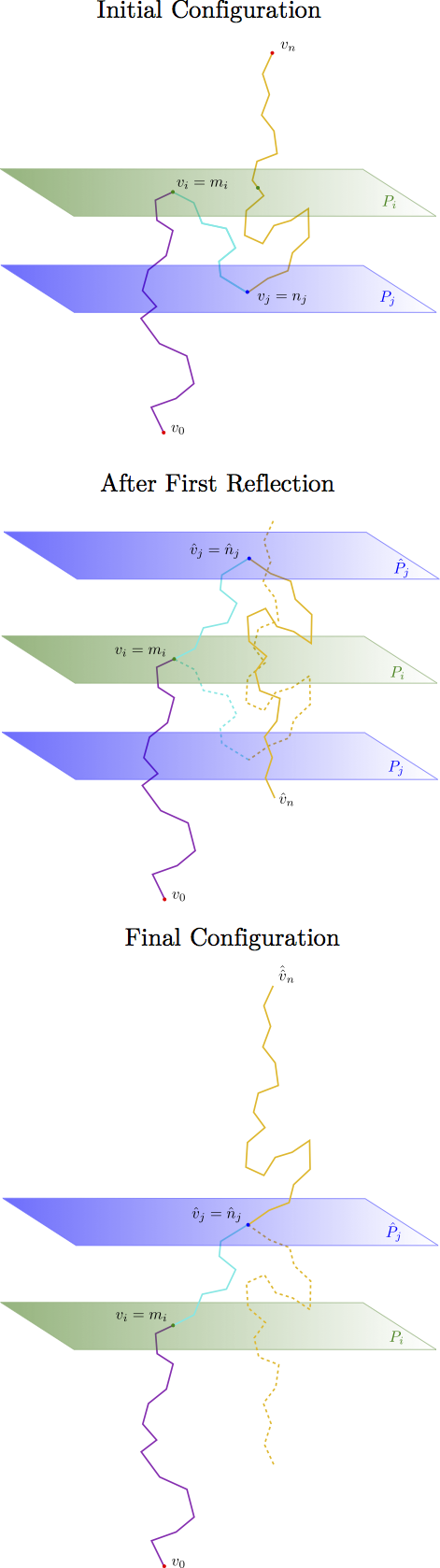}
  \end{center}
 \caption{Here we have a ``typical'' double reflection move for part two of the proof. First, we identify the local maxima, $m_i$, and local minima, $n_j$. The first reflection through the plane $P_i$ reflects the segments between $v_i$ and $v_n$, resulting in new segments between the points $\hat{v}_i, \hat{v}_{i+1},...,\hat{v}_n$. Then the reflection through $\hat{P}_j$, the reflected image of $P_j$, reflects the segments between $\hat{v}_j$ and $\hat{v}_n$. This gives us our new configuration consisting of the points $v_0, ..., v_i, \hat{v}_{i+1}, ..., \hat{v}_j,\hat{\hat{v}}_{j+1}, ..., \hat{\hat{v}}_n$.}
    \label{dr}
\end{figure}

Using the same argument as with single reflections through planes on the convex hull, $\hat{\phi_i} \geq \phi_i$ and $\hat{\phi_j} \geq \phi_j$. As in the previous proof, each of these segments has been changed by an isometry, the only possible reduction in tube radius will come from inter-segmental interactions. We refer to the three subsegments as $S_0$, from $v_0$ to $v_i$, $S_i$, from $v_i$ to $\hat{v}_j$, and $S_j$, from $\hat{v}_j$ to $\hat{\hat{v}}_n$. We also define $d_{(i,j)}$ to be the minimum distance between $P_i$ and $P_j$. 

The interaction between $S_0$ and $S_i$ is what it would be for a single reflection through a point on the convex hull, as above. Thus, interaction between $S_0$ and $S_i$ does not decrease the tube radius this walk can accommodate. Likewise for $S_i$ and $S_j$. (We note that, as with the single reflections, these reflections are in a neighborhood of valid choices.) All that remains are the potential interactions between $S_0$ and $S_j$. These segments are separated by a slab, the space between $P_i$ and $P_j$. We also note that $S_j$ is the original segment from $v_{j}$ to $v_n$ translated by $d_{(i,j)}$ perpendicular to $P_j$, as two parallel reflections are a translation.   

Now we consider two points, $a \in s_h \subseteq S_0$, and $b \in s_k \subseteq S_j$. Suppose toward contradiction that after this double reflection, $a$ and $\hat{\hat{b}}$ violate the long range interaction and $d(a,\hat{\hat{b}})<2r$ and $a$ and $\hat{\hat{b}}$ are a doubly critical pair. In order for $a$ and $\hat{\hat{b}}$ to become doubly critical, there is a minimum number of edges that must be between between $s_h$ and $\hat{\hat{s}}_k$, namely, half the number of edges in a regular polygon with interior angles greater than or equal to $2\arctan(2r).$ Regular polygon geometry leads to this number being $\left\lfloor \ds \frac{\pi}{\pi-2\arctan(2r)}\right\rfloor.$ We conclude that there must be at least this many edges between $s_h$ and ${s}_j$. 

\begin{figure}[!htp]
  \begin{center}
    \includegraphics[scale=1]{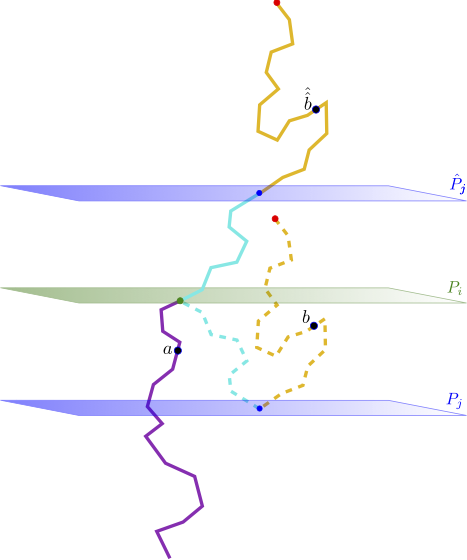}
  \end{center}
 \caption{Because $a$ is always located below the plane $P_i$, $d(a,b)<d(a,\hat{\hat{b}})$ following such a double reflection.}
    \label{double}
\end{figure}

We must also have that $d(a,{{b}})<2r,$ as the translation moves $S_j$ away from $S_0$ as in Figure \ref{double}. Because the initial configuration can accommodate a tube of radius $r$, and $a$ and $b$ are within $2r$ of each other, they are therefore not a pair of doubly critical points, and there are no pairs of doubly critical points between them. We know, therefore, that for every point $x$ along the segment between $a$ and $b$, the function $d(a,x)$ is strictly increasing. Likewise for the $d(b,x)$ for $x$ along the segment from $b$ to $a$. This confines the segment between $a$ and $b$, and it must be contained in the union of the two spheres of radius $d(a,b)$ centered at $a$ and $b$. We will consider how many edges can fit between $a$ and $b$ in this union, as in Figure \ref{spheres}.  

The most edges one can fit into this union of spheres without violating the short range (angle) restriction or having a doubly critical pair will be a segment of a regular $n$-gon where the internal angle $\theta$ is greater than or equal to $2\arctan(2r)$. The diagonal of this arc, which we will call $d_m$, must be less than $d(a,b)$ (and therefore also less than $2r$) in order to be a segment between $a$ and $b$. From regular polygon geometry, we know that $d_m = d\sin\left(\ds \frac{\pi m}{n}\right)$ where $m$ is the number of edges connecting the endpoints of the diagonal determined by $d_m$, and $d$ is the diameter of the $n$-gon. We can show that the diameter of this $n$-gon is $\sqrt{4r^2+1}$ and that $n = \frac{2\pi}{\pi-\theta}$. Then we find 
\begin{eqnarray*}
m& < & \ds \frac{2}{\pi - 2\arctan(2r)}\arcsin\left(\ds \frac{2r}{\sqrt{4r^2+1}}\right)\\
& <  & \ds \frac{2}{\pi - 2\arctan(2r)}\arctan(2r).\\
\end{eqnarray*}

Therefore, we know that the number of edges between $s_h$ and $s_k$ is $m$ and that $$\ds\frac{2\arctan(2r)}{\pi-2\arctan(2r)} >m>\left\lfloor \ds \frac{\pi}{\pi-2\arctan(2r)}\right\rfloor.$$
This is a contradiction, as $2\arctan(2r) \not> \pi$ for any $r$. Therefore no such doubly critical pair can exist. We conclude that the long range interaction between $S_0$ and $S_j$ does not increase the tube radius. We also conclude that because the segments in $S_0$ and $S_j$ are sufficiently far apart to not form a doubly critical pair with room to spare, they are also sufficiently far apart that our double reflection is in a neighborhood of valid double reflection choices. Again, this fact is used in the proof or ergodicity.   

We may proceed in this manner, always with respect to the original axis $d^*$, reducing the number of local maxima and local minima until our configuration is strictly increasing relative to the axis $d^*$. The third and last step is to straighten this configuration. 
\begin{figure}[!htp]
  \begin{center}
    \includegraphics[scale=.75]{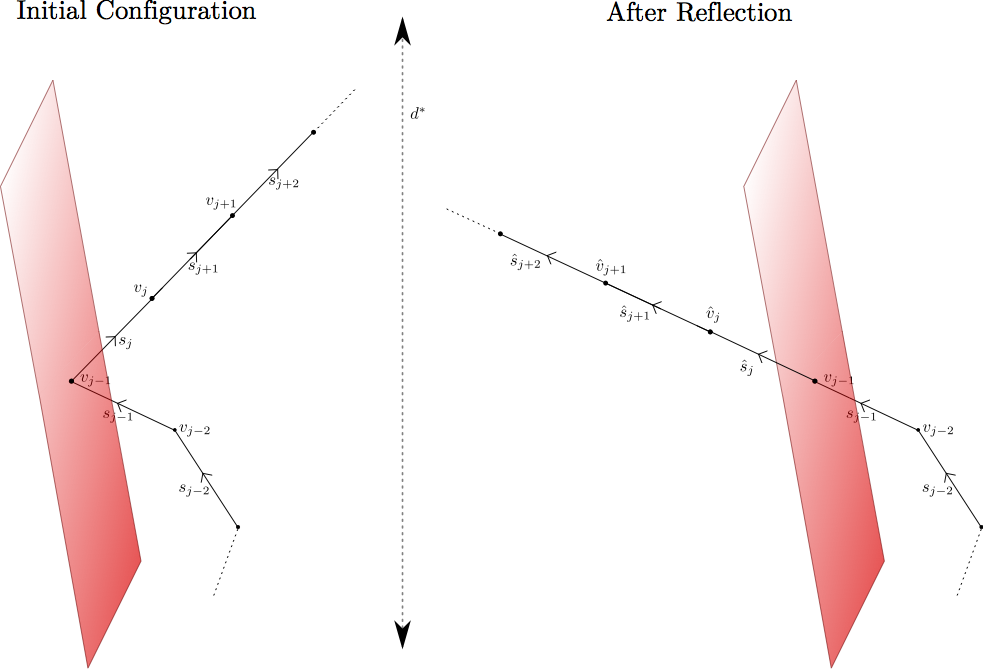}
  \end{center}
 \caption{Via a single reflection reflection, we straighten the segments $s_{j+1}, ..., s_n$ to be in line with $s_j$.}
    \label{straight}
\end{figure}
Suppose we have a configuration as in Figure \ref{straight}, where the segments $s_{j+1}, ..., s_n$ have been straightened to be in line with $s_j$. We preform a single reflection as in Figure \ref{straight} that results in the segments $s_{j+1}, ..., s_n$ being in line with $s_j$. The only remaining question is does this reflection decrease the radius?

\begin{figure}[!htp]
  \begin{center}
    \includegraphics[scale=.75]{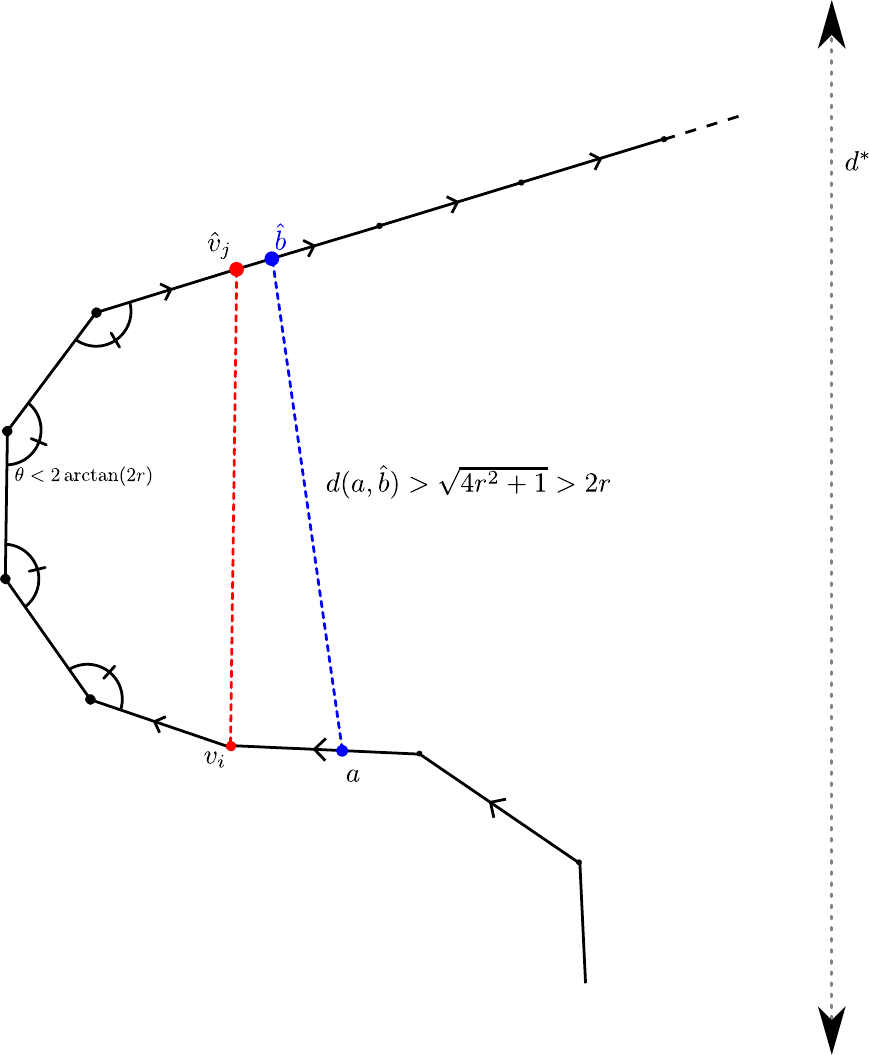}
  \end{center}
 \caption{The closest $a$ and $\hat{b}$ can get is the diameter of a regular, planar $n$-gon with internal angles $\theta$ greater than or equal to $2\arctan(2r)$, which is  }
    \label{sproof}
\end{figure}

Because the new angle $\phi_i  = \pi \geq 2\arcsin(2r)$, there is certainly no violation of the short range interaction. Now suppose that we have some doubly critical pair, $a$ and $\hat{b}$. In order to be doubly critical, and for the configuration is strictly increasing relative to the axis $d^*$, the closest $a$ and $\hat{b}$ can get is the diameter of a regular, planar $n$-gon where the internal angle $\theta$ is greater than or equal to $2\arctan(2r)$, as in Figure \ref{sproof}. Therefore the closest $a$ and $\hat{b}$ can be is $\sqrt{4r^2 +1}< 2r.$ We conclude that this reflection does not decrease the tube radius, and that these reflections each have a neighborhood of valid reflection choices.  

Therefore, we may straighten our segment until it is the straight configuration. This concludes the proof that the reflection method is transitive of the space of thick walks, $\W_{(n,r)}$, as we can move from any configuration to the straight configuration in a finite number of moves.

\end{proof}

\begin{lemma}
\(F\) being forward accessible and \(X\) having a probability density function implies that \(F\) is a \(T\)-chain.
\end{lemma}

This is a modification of the argument from {\it Markov Chains and Stochastic Stability} by Meyn and Tweedie \cite{meyn}. We mainly use Propositions 7.1.5 and 6.2.4 from their book \cite{meyn}.
 
\begin{lemma}
A \(T\)-chain with a reachable state is bounded in probability on average if and only if it is positive Harris recurrent.
\end{lemma}
This is proposition 18.3.2 from {\it Markov Chains and Stochastic Stability} by Meyn and Tweedie \cite{meyn}. 

\begin{lemma}
\(F\) is aperiodic.
\end{lemma}
Seeking a contradiction, suppose there exists \(C_i\) with \(i\in \Z_n\) which are disjoint non-empty closed sets and the probability of going from \(C_i\) to \(C_{i+1}\) is one. Let \(w \in C_i\). If we apply a pair of reflections at the second to last vertex through planes containing the last edge, then \(w\) is left fixed by this double reflection. Thus, there exists a double reflection move \(r\) with \(r(w) = w \notin C_{i+1}\). Since \(C_{i+1}\) is closed, there is an open neighborhood of \(w\) which is not in \(C_{i+1}\) so there is a positive probability of landing in this open neighborhood of \(w\). This contradicts the requirements of periodicity. Thus, \(F\) is aperiodic.

\begin{theorem}
The Markov chain defined above is ergodic.
\end{theorem}

\begin{proof}
First, the first lemma shows that for any walk there is a sequence of finitely many double reflections which bring this walk to the straight walk, and each double reflection in this sequence is on the interior of a smooth section of \(F\). Then, by concatenating this sequence with the reverse of another such sequence, we can connect any two walks on the interior of \(\W_{(n,r)}\) with a sequence of reflections on the interior of a smooth section of \(F\). This means that \(F\) creates a forward accessible Markov chain. Next, since our noise parameter \(X\) has a probability density function, namely a constant, we can combine this with forward accessibility to get that the Markov chain is a T-chain. Since our state space is compact, we get that any sequence of probability distributions is tight, so in particular our Markov chain \(F\) is bounded in probability on average. The fact that every walk can be brought to any neighborhood of the straight walk with positive probability means that \(F\) has a reachable state and so \(F\) being a \(T\)-chain with a reachable state which is bounded in probability on average makes \(F\) positive Harris recurrent. Finally, since there is a fixed point at the interior of a smooth section, \(F\) cannot be periodic and so the aperiodic ergodicity theorem tells us that \(F\) is ergodic.
\end{proof}

\section{Knotting In Open Chains}

\subsection{Definition of Knotting in Open Chains}

Methods for determining the nature of knottedness in open chains relies on closing an open configuration and determining the knot type of that closed arc by calculating the knot polynomials for the closed configuration \cite{millett2005}. These methods differ in the method of closing the open configuration. One possible method is direct closure: the two ends of the open chain are connected by a straight arc. However, there are many configurations that we want to be considered as knots that direct closure fails to detect, like the one in Figure \ref{direct}.

\begin{figure}[!htp]
  \begin{center}
    \includegraphics[scale=.4]{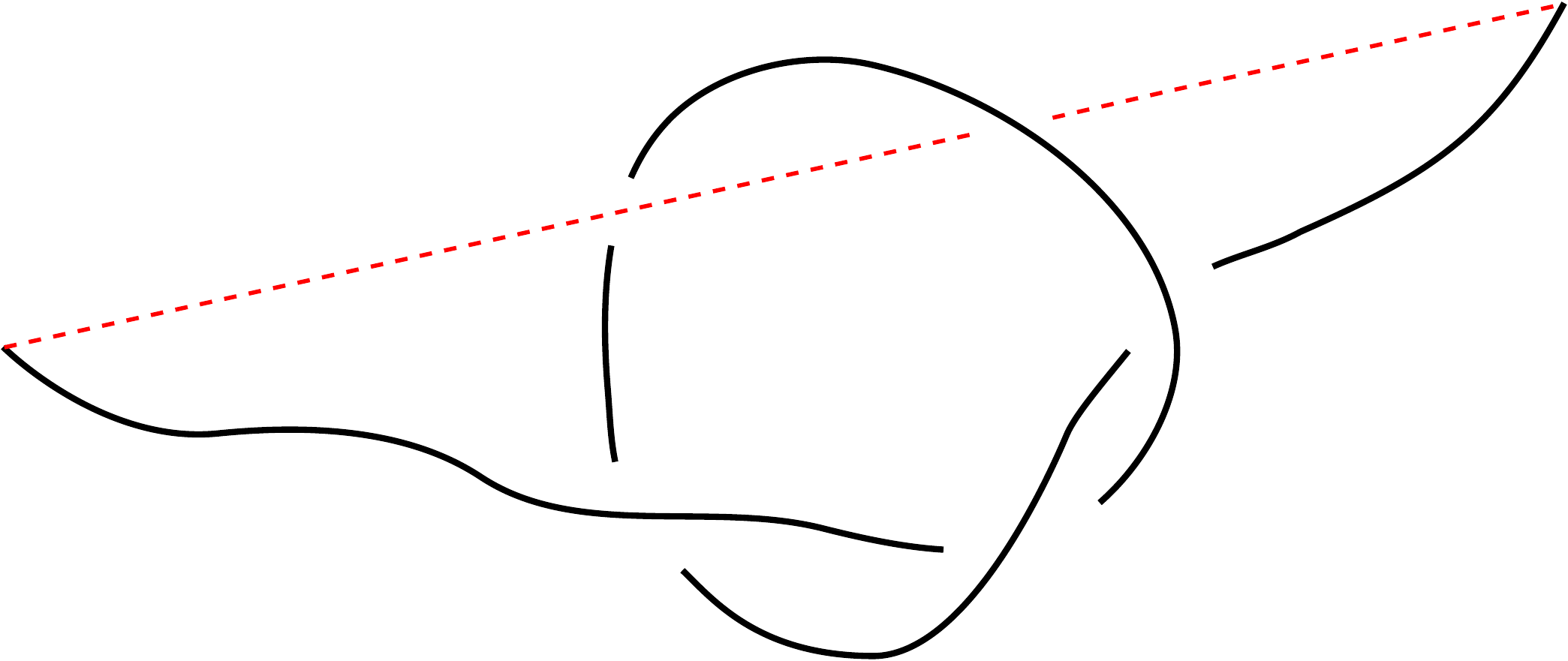}
  \end{center}
 \caption{A possible configuration where the direct closure method would give an unknot.}
    \label{direct}
\end{figure}

Taylor described a method where the ends of the open chain were fixed, and the remaining configuration is deformed and smoothed as much as possible \cite{taylor}. Following these deformations, the ends are closed to a distant point (often referred to as closure at infinity) and the knot type is determined via knot polynomials \cite{taylor}. This closure scheme proves to be problematic: Millett et al. discovered configurations which when smoothed with Taylor's algorithm from head to tail produced unknots, and when smoothed from tail to head produced trefoils \cite{Ken05}. We would prefer a method where the determination of the knot type is independent of the ordering of the vertices.

These ambiguities reveal that in open chains, knottiness is not deterministic. This suggests that a probabilistic measure of how knotted each chain should be used instead. Mansfield proposed a method where the ends of the protein structure were extended, by straight segments, to two random points on a large sphere enclosing the protein \cite{Mansfield94, Mansfield97}. These ends were connected by the arc of a great circle connecting them and the knot type of this closed configuration determined using knot invariants \cite{Mansfield94, Mansfield97}. Mansfield performed this operation 100 times for each protein to determine the dominant knot type in the resulting data \cite{Mansfield94, Mansfield97}.

Millett et al. proposed a simplified variant of Mansfield's method: the open arc is situated inside of a large sphere, and the ends of the arc are connected to a point on the sphere by two straight segments \cite{Ken05}. This closure produces some knot type. This procedure is repeated for points over the surface of the sphere, so that for a collection of points on the surface of the sphere there is a specific knot   and a spherical distribution of knot types is created associated with the open knot, sometimes referred to as the knot spectrum of the configuration, and from this the dominant knot type is determined \cite{Ken05}. For several examples from the protein database, Millett's method is in agreement with Mansfield's for all samples and differed from Taylor's for at least one \cite{millett2005}.
   
\subsection{Method of Calculation}

For our uses, we will want to associate a single knot type with an open configuration, rather than a knot spectrum. This requires determining the dominant knot type. We will use three standards for dominance, from \cite{millett2005}:
\begin{itemize}
\item {\bf Strong dominance} requires that the knot type occurs in $90\%$ or more closure. 
\item {\bf Dominance} requires that the knot type occurs more than twice as often as the second likeliest knot type. For example, if trefoils occurred for $70\%$ of the closures and the unknot occurred for $30\%$ of the closures, we would say that trefoil was the dominant knot type. 
\item {\bf Weak dominance} requires that the knot type occurs for a majority of closures.  
\end{itemize}

For length $300$ walks, the average dominance percentage was $94\%$ with a standard deviation of less than $1\%$, although we use weak dominance for our criterion when determining knot type. For the shorter chains, the dominance percentages were higher, as in Millett's work \cite{millett2005}. Similarly, thicker chains had higher percentages for the dominant knot type.

\section{Numerical Method and Results}
\subsection{Methods}
We generated a data set consisting of 5000 samples for each of the following pairs of lengths and thicknesses: 
\begin{itemize}
\item lengths: 100 to 1000 in steps of 100. 
\item thicknesses: 0, 0.1, 0.2, 0.3, 0.4, 0.5, 0.6, 0.7, 0.8, 0.9 and 1 (Note, these correspond to minimal bending angles of about $0^\circ$, $23^\circ$, $44^\circ$, $62^\circ$, $77^\circ$, $90^\circ$, $100^\circ$, $109^\circ$, $116^\circ$, $122^\circ$, and $127^\circ$.)
\end{itemize}

From this data we examined squared radius of gyration, squared end to end distance, and for the length 300 samples, knotting in the open chains. 

\subsection{Runtime and Probability of Acceptance}

For the above simulations, we recorded the probability of accepting a move and the average run time for each length and thickness. Below in Table \ref{acceptance}, we have the probability for accepting a single or double reflection move, for a variety of lengths and thicknesses. When the thickness is 0, every move is accepted and the acceptance rate is 100\%. Increasing length or thickness decreases the probability of accepting a move, though no combination we used resulted in an acceptance rate less than 25\%.  

\begin{table}[h]
\centering
\begin{tabular}{|c|c|c|c|c|c|c|c|c|c|c|c|}
\hline
\rowcolor[HTML]{C0C0C0} 
                                     & \textbf{0.0} & \textbf{0.1} & \textbf{0.2} & \textbf{0.3} & \textbf{0.4} & \textbf{0.5} & \textbf{0.6} & \textbf{0.7} & \textbf{0.8} & \textbf{0.9} & \textbf{1} \\ \hline
\cellcolor[HTML]{C0C0C0}\textbf{100} & 100.00       & 78.56        & 68.47        & 60.71        & 54.12        & 48.51        & 43.58        & 39.39        & 35.79        & 32.65        & 29.95      \\ \hline
\cellcolor[HTML]{C0C0C0}\textbf{200} & 100.00       & 73.57        & 63.89        & 56.73        & 50.71        & 45.55        & 41.15        & 37.35        & 34.07        & 31.25        & 28.78      \\ \hline
\cellcolor[HTML]{C0C0C0}\textbf{300} & 100.00       & 70.68        & 61.25        & 54.39        & 48.74        & 43.84        & 39.65        & 36.08        & 32.96        & 30.31        & 27.99      \\ \hline
\cellcolor[HTML]{C0C0C0}\textbf{400} & 100.00       & 68.55        & 59.37        & 52.79        & 47.32        & 42.68        & 38.53        & 35.19        & 32.20        & 29.64        & 27.39      \\ \hline
\cellcolor[HTML]{C0C0C0}\textbf{500} & 100.00       & 66.98        & 58.02        & 51.54        & 46.29        & 41.70        & 37.84        & 34.46        & 31.60        & 29.10        & 26.94      \\ \hline
\cellcolor[HTML]{C0C0C0}\textbf{600} & 100.00       & 65.74        & 56.88        & 50.69        & 45.35        & 41.02        & 37.18        & 33.89        & 31.09        & 28.63        & 26.57      \\ \hline
\cellcolor[HTML]{C0C0C0}\textbf{700} & 100.00       & 63.65        & 56.11        & 49.79        & 44.94        & 41.23        & 37.13        & 33.14        & 30.59        & 28.62        & 25.99      \\ \hline
\cellcolor[HTML]{C0C0C0}\textbf{800} & 100.00       & 63.11        & 55.33        & 49.79        & 44.26        & 39.28        & 36.50        & 32.79        & 30.29        & 27.88        & 25.58      \\ \hline
\cellcolor[HTML]{C0C0C0}\textbf{900} & 100.00       & 63.16        & 54.17        & 48.56        & 43.26        & 39.49        & 35.62        & 32.57        & 30.12        & 27.57        & 25.55      \\ \hline
\cellcolor[HTML]{C0C0C0}\textbf{1000} & 100.00       & 62.73        & 53.28        & 48.06        & 42.75        & 38.94        & 34.94        & 32.42        & 29.70        & 27.79        & 25.59      \\ \hline
\end{tabular}   \caption{The acceptance probability of reflection moves for a variety of thicknesses and lengths. Note that the longer and thicker a chain, the less likely a move will be accepted, though none of these acceptance rates is less that 25\%. }
   \label{acceptance}
\end{table}

For the pivot method, the probability of accepting a move is about $N^{-0.19}$ for a walk of length $N$ \cite{madrassokal}. If the probability of accepting a reflection move also scales like $N^\alpha$, one can solve for $\alpha$ as a function of thickness, $r$. The results are shown in Figure \ref{accept}.

The scaling of our acceptance rates is considerably closer to $0$ than the pivot method for all $r$. At $\alpha = 0$, every reflection move is accepted, and run time would be minimal. Thus, the closer to zero, the higher the probability of acceptance. Our acceptance probability is particularly good for very thin and very thick chains.

We also calculated average run time. While this will vary depending on machine specifics and implementation, the scaling of run time is still meaningful. We assumed that run time would also be proportional to $N^{\beta}$ and solved for $\beta$ as a function of $r$. 

Similar to our scaling for the probability of accepting a reflection, the run time scales like $N^{2.75}$ for chains with no thickness, and less than that for chains with thickness greater than 0.4. When $r\in [0.1,0.3]$, generation is more computationally expensive due to the combination of flexibility and thickness resulting in more interactions between segments.
\begin{center}
\begin{figure}[!htp]
\centering
    \includegraphics{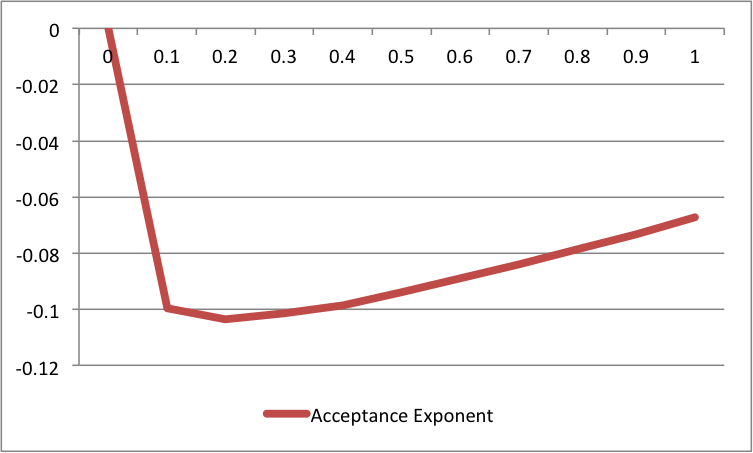}
    \caption{We analyzed how the acceptance probability of a reflection move scaled as a function of $N$. Assuming that the probability of acceptance is proportional to $N^{\alpha}$, we solved for $\alpha$ and plotted it as a function of thickness.}
  \label{accept}
\end{figure}
\end{center}

\begin{center}
\begin{figure}[!htp]
\centering
    \includegraphics{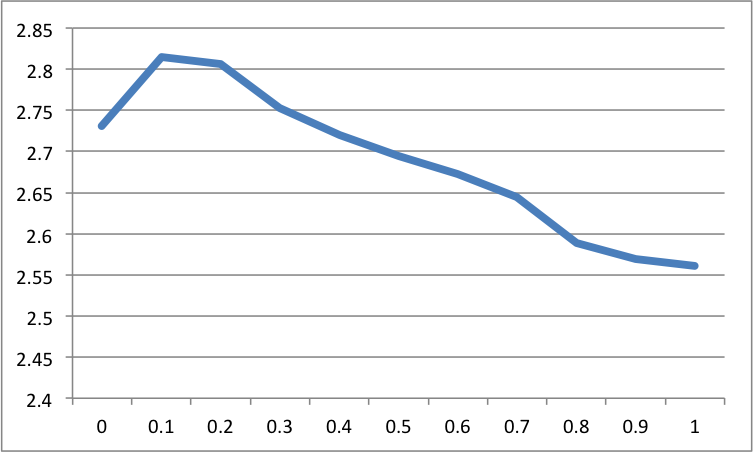}
    \caption{We analyzed how the generation run time scaled as a function of $N$. Assuming that the run time is proportional to $N^{\beta}$, we solved for $\beta$ and plotted it as a function of thickness.}
  \label{runtime}
\end{figure}
\end{center}

\section{Results}
\subsection{Knot Probability and Thickness}

 Previous studies have shown that for closed rings of DNA, where effective thickness has been increased through electrostatic repulsion, thinner (less self repelling chains of DNA) had a higher incidence of trefoil and knot formation than their thicker counterparts \cite{rybenkov}. Likewise, previous work on and off the simple cubic lattice have shown that knot formation decreases precipitously as thickness increases \cite{frank}.

\begin{figure}[!htp]
  \begin{center}
     \includegraphics{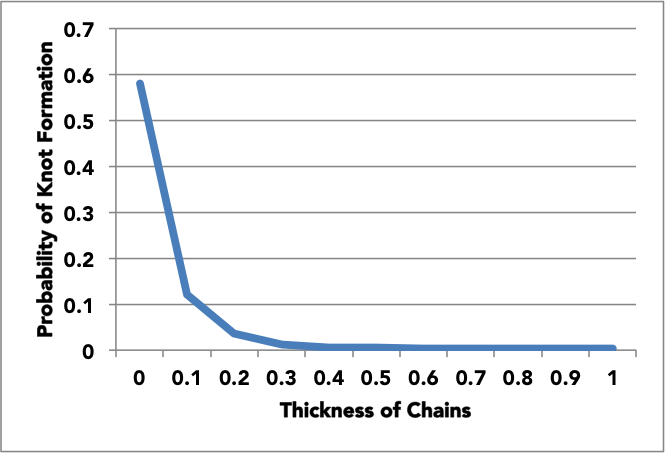}
    \caption{Knot probability as a function of thickness for self avoiding walks of length 300. For each $0.1$ added to thickness, probability of knot formation his cut by 50\% or more, up to $r=0.5$ where probability of knot formation stabilizes at less than $0.5\%$. }
\label{knot_prob}
  \end{center}
\end{figure}

In Figure \ref{knot_prob} we see that for $r \in [0,0.4]$, an increase of $0.1$ in thickness results in the probability of knotting being reduced by 50\% or more, for walks of length 300. In this range, very modest changes to thickness result in a dramatic decrease in the probability of knot formation. For $r > 0.4$ we see different behavior: the probability of knotting is largely unchanged by increases to thickness. While we suspect that this is an artifact of the short length scale, it is evidence that there are ranges where knotting depends more critically on thickness. This confirms that the probability of knot formation is strongly effected by the effective thickness of the chain, and is consistent with Deguchi and Tsurusaki's work analyzing knot formation in self avoiding walks using the rod and bead model \cite{universal_knotting}. We expect to see similar results for longer length scales, taking into account the increased probability of knot formation as length increases \cite{deguchi1994}.

\begin{figure}[!htp]
  \begin{center}
     \includegraphics[scale=.75]{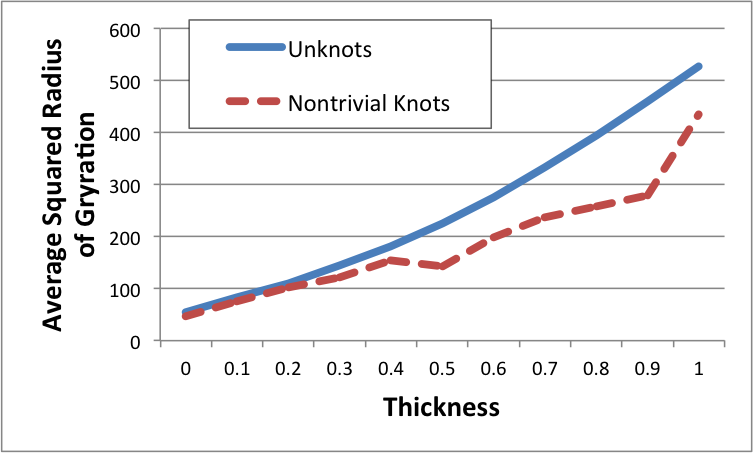}
    \caption{Knot size (average squared radius of gyration) as a function of thickness for self avoiding walks of length 300. Unknots (blue and solid) are larger on average than nontrivial knots (dashed red).}
\label{knot_size}
  \end{center}
\end{figure}

In Figure \ref{knot_size} we see that, as expected, the thicker chains have much larger squared radius of gyration. Also, as expected, the unknotted configurations are larger than the knotted configurations for all thicknesses, although the difference between the knotted population and the unknotted population is more dramatic for the thicker chains. The variance in average squared radius of gyration in the thick, knotted conformations is due to the small number of knotted samples for $r\geq0.4$.  

\subsection{Growth Exponents and Regime Change}

As we would expect, just as average squared radius of gyration increases with the length of a random walk, with or without thickness, average squared radius of gyration increases dramatically with the introduction and continued increase of thickness, as in Figure \ref{RG}. 

\begin{figure}[htbp]
   \centering
   \includegraphics[scale=.65]{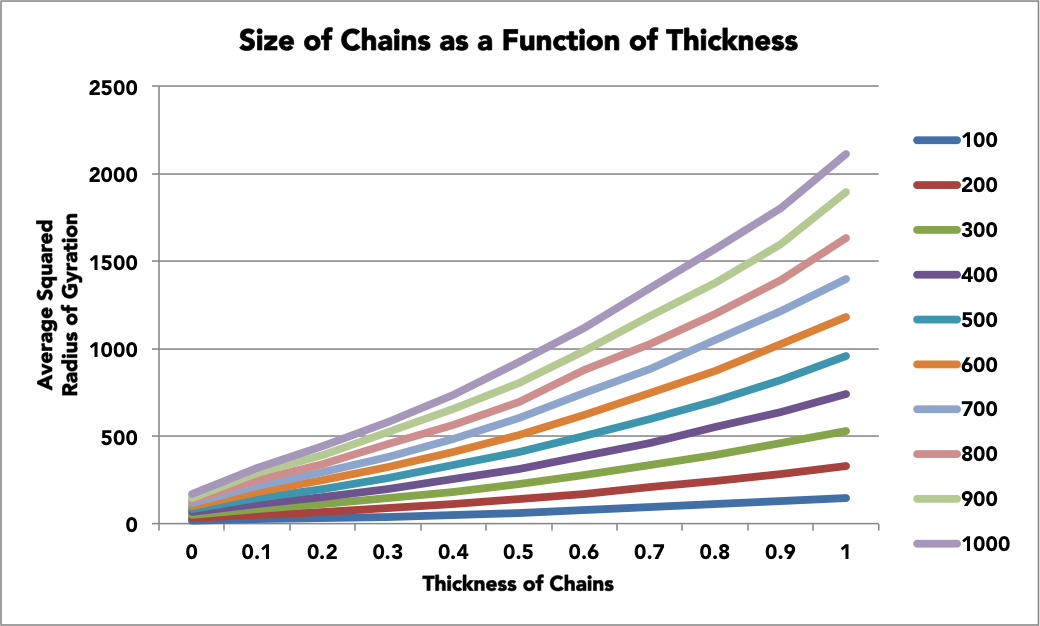} 
   \caption{A comparison of average squared radius of gyration as a function of thickness for self avoiding walks of varying lengths: 100, 200, 300, 400, 500, 600, 700, 800, 900 and 1000.} 
   \label{RG}
\end{figure}

As discussed in the introduction, polymer models are often classified in terms of the characteristics of the solvents in which they are submersed: 
\begin{itemize}
\item {\bf ``Theta solvent'' or ``Theta temperature'' models:} These demonstrate growth of squared end to end distance and squared radius of gyration proportional to $N$. Ideal chains fall into this category. 
\item {\bf ``Good solvent'' models:} These models demonstrate growth of squared end to end distance and squared radius of gyration proportional to $N^{\nu}$ with $\nu > 1$. $\nu$ has been estimated to be $1.2$, and in the simple cubic lattice model (SAW) the exponent has been shown to be $1.18$, and Edwards' perturbation calculation of $1.176 \pm .002$ \cite{degennes, micheletti2011, madras_ergo, DoiEdwards}.  
\item {\bf ``Bad solvent'' models:} These models demonstrate growth of squared end to end distance and squared radius of gyration proportional to $N^{\nu}$ with $\nu < 1$. 
\end{itemize}

We found the average squared radius of gyration of each population, and estimated the exponent $\nu$ for each thickness we sampled. 

\begin{figure}[htbp]
   \centering
   \includegraphics[scale=.65]{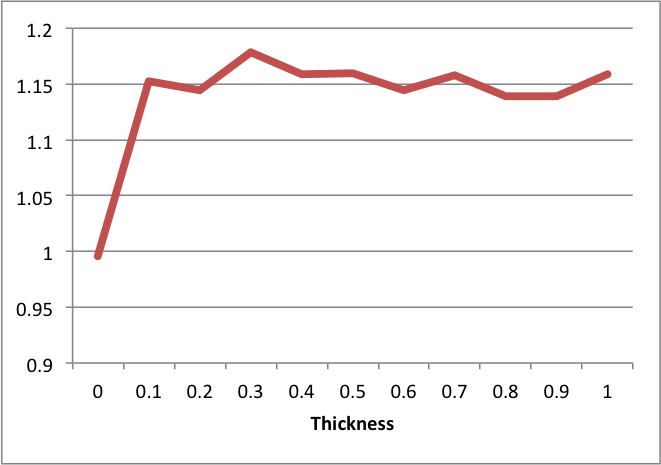} 
   \caption{We analyzed how the squared radius of gyration scaled as a function of $N$. Assuming that $RG_n^2 \propto N^{\nu}$, we solved for $\nu$ and plotted it as a function of thickness for self avoiding walks (red, solid). The exponent, as expected for the idea case, was close to $1$ for thickness $0.0$ and increased to $1.16$ where it was fairly stable ($\pm .01$) for thicknesses $r\in[0.2,1.0]$.} 
   \label{gexp}
\end{figure}

For ideal chains it had been shown that $RG_n^2 \propto N$, and it is easily showed that for $r = \infty$, we have the straight configuration where $RG_n^2 \propto N^2$ \cite{zirbel}. It is therefore natural to characterize some $\nu$ for each thickness $r \in [0,\infty)$ such that $RG_n^2 \propto N^{\nu}$. Excluded volume has been classically characterized, via mean field theory, by the estimate $\nu = 1.2$ \cite{DoiEdwards}. In wet experiments and numerical simulation for good solvent polymers, $\nu$ has been estimated as everything between $1.1$ and $1.2$, with $2*0.588 = 1.18$ in walks on the simple, cubic lattice,  $1.176$ with the Edwards model and with the reflection method on the simple cubic lattice \cite{DoiEdwards, hyper}. We can see that there is an immediate impact of thickness: the scaling exponent increases to the $[1.14,1.17]$ interval almost immediately, as predicted by Vologodskii's simulations of very thin walks and rings \cite{vologodskii}.

For comparison, we will refer to $\nu_r$ as the scaling exponent for the self avoiding walk with thickness $r$, and $\mu_r$ as the scaling exponent for the short range only walk of thickness $r$.

We found the average $RG_n^2$ for each data set, and did linear regression with vertical offsets on the $\log$ of the data. This allowed us to solve for $\nu$ as a function of thickness. 

\begin{figure}[htbp]
   \centering
   \includegraphics[scale=.65]{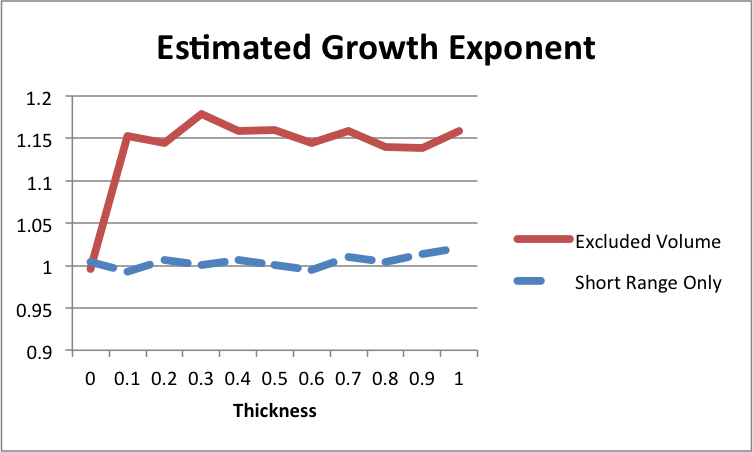} 
   \caption{We analyzed how the squared radius of gyration scaled as a function of $N$. Assuming that $RG_n^2 \propto N^{\nu}$, we solved for $\nu$ and plotted it as a function of thickness for self avoiding walks (red, solid) and short range constraint only walks (blue, dashed). The exponent, as expected for the idea case, was close to $1$ for thickness $0.0$ and increased to $1.16$ where it was fairly stable ($\pm .01$) for thicknesses $r\in[0.2,1.0]$. For short range only walks, the exponent was consistently close to 1, as expected.} 
   \label{gexp}
\end{figure}

A welcome result of our study was the maximum value that $\nu_r$ obtains - about $1.16$. This estimate of the growth exponent is close (though lower) than previous numerical and scientific experiments, including Madras and Sokal's exponent for self avoiding walks on the lattice, $1.18$, Edwards' peturbation calculation of $1.176 \pm .002$ and Caraciolo, Ferraro and Pilisetto's $1.1723 \pm 0.005$ \cite{needle, cotton, DoiEdwards, madras_ergo}.

These results suggest that the effect of excluded volume instantaneously effects size and the scaling of size: even very, very thin chains are dramatically swelled by the excluded volume effect. 

\section{Conclusions}

The numerical results from the reflection method have confirmed a great deal about random walks with excluded volume. Thick walks have a greater radius of gyration, and the difference between knotted and unknotted populations is more dramatic for thick walks. Thick walks are more likely to be unknotted than thin walks of the same length, and for all thicknesses, the longer the walk, the higher the probability of knotting. 

At the same time, we have some new facts as well. For thin walks, those with $r < 0.4$, small increases in thickness result in dramatic decreases in knot probability: increasing $r$ by $0.1$, or one tenth of the edge length, reduces knot probability by about half. But for walks with $r \geq 0.4$, knot probability is not decreased when the radius is increased. What is more remarkable is that this behavior seems to be independent of length, suggesting that this is a true behavioral phase change at $r = 0.4$. 

Importantly, our examination of scaling as a function of thickness has shown that the exponent estimation associated with excluded volume is close to $1.16$, as in walks on the simple cubic lattice. We plan on investigating how these exponents vary for other polymer models, which may include an analysis of this model where long range and short range are allowed to vary independently. We conjecture that for sufficiently long lengths, the scaling exponent will converge to a constant for $r \in (0,\infty)$.   

There is much to do to fully explore the capabilities of this algorithm, including at a very basic level examining longer lengths, very thin chains and very thick chains to confirm that trends described above persist. We will also examine how knot probabilities and knot sizes scale with the length $N$,  as in Deguchi and Tsurusaki's work, and examining knot length as a function of walk thickness \cite{universal_knotting}. Beyond the realm of thick walks, another goal is to extend the similar results for rings to greater lengths, as well as simulations of polymer melts of thick chains and rings \cite{kyle}.

\bibliographystyle{plain}
\bibliography{myrefs.bib}

\end{document}